\documentclass[12pt,a4paper,leqno]{amsart}

\title[Propagation of polynomial phase space singularities]{Propagation of polynomial phase space singularities for Schr\"odinger equations with quadratic Hamiltonians}

\author[P. Wahlberg]{Patrik Wahlberg}

\address{Department of Mathematics, Linn{\ae}us University, SE--35195 V{\"a}xj{\"o}, Sweden}

\email{patrik.wahlberg@lnu.se}

\usepackage{latexsym}
\usepackage{amsmath}
\usepackage{amssymb}
\usepackage{amsthm}
\usepackage{amsfonts}
\usepackage{mathrsfs}
\usepackage{calc}
\usepackage{cite}
\usepackage{color}

\numberwithin{equation}{section}          

\newtheorem{thm}{Theorem}
\numberwithin{thm}{section}

\newcommand{\rubrik}{}
\newtheorem{prop}[thm]{Proposition}
\newtheorem{cor}[thm]{Corollary}
\newtheorem{lem}[thm]{Lemma}

\theoremstyle{definition}

\newtheorem{defn}[thm]{Definition}

\theoremstyle{remark}

\newtheorem{rem}[thm]{Remark}              

\newcommand{\pd}[1] {\partial ^#1}

\newcommand{\ro}{\mathbb R}
\newcommand{\no}{\mathbb N}
\newcommand{\rr}[1]{\mathbb R^{#1}}
\newcommand{\nn}[1]{\mathbb N^{#1}}

\newcommand{\co}{\mathbb C}
\newcommand{\cc}[1]{\mathbb C^{#1}}

\newcommand{\Ker}{\operatorname{Ker}}

\newcommand{\ep}{\varepsilon}
\newcommand{\fy}{\varphi}

\newcommand{\supp}{\operatorname{supp}}

\newcommand{\eabs}[1]{\langle #1\rangle}

\newcommand{\Sp}{\operatorname{Sp}}
\newcommand{\ssp}{\operatorname{sp}}
\newcommand{\Mp}{\operatorname{Mp}}

\newcommand{\charac}{\operatorname{char}}

\newcommand{\cS}{\mathscr{S}}

\newcommand{\cK}{\mathscr{K}}

\newcommand{\wh}{\widehat}
\newcommand{\re}{{\rm Re} \, }
\newcommand{\im}{{\rm Im} \, }
\newcommand{\J}{\mathcal{J}}
\def\la{\langle}
\def\ra{\rangle}

\begin{document}

\begin{abstract}
We study propagation of phase space singularities for a Schr\"odinger equation with a Hamiltonian that is the Weyl quantization of a quadratic form with non-negative real part. 
Phase space singularities are measured by the lack of polynomial decay of given order in open cones in the phase space, which gives a parametrized refinement of the Gabor wave front set.  
The main result confirms the fundamental role of the singular space associated to the quadratic form for the propagation of phase space singularities. 
The singularities are contained in the singular space, and propagate 
in the intersection of the singular space and the initial datum singularities along the flow of the Hamilton vector field associated to the imaginary part of the quadratic form. 
\end{abstract}


\maketitle

\section{Introduction}

We study the initial value Cauchy problem for equations of Schr\"odinger type
\begin{equation*}
\left\{
\begin{array}{rl}
\partial_t u(t,x) + q^w(x,D) u (t,x) & = 0,  \qquad t \geqslant 0, \quad x \in \rr d, \\
u(0,\cdot) & = u_0,  
\end{array}
\right.
\end{equation*}
where $u_0$ is a tempered distribution on $\rr d$, $q=q(x,\xi)$ is a quadratic form on the phase space $(x,\xi) \in T^* \rr d$ with $\re q \geqslant 0$, and $q^w(x,D)$ is a Weyl pseudodifferential operator. 
This family of equations contains as particular cases the proper Schr\"odinger equation where $q(x,\xi)= i |\xi|^2$, the harmonic oscillator where $q(x,\xi)= i (|x|^2 + |\xi|^2)$ and the heat equation where $q(x,\xi)= |\xi|^2$.

We show results on the propagation of a parametrized version of the Gabor wave front set.
The parametrized notion is called the $s$-Gabor wave front set (or singularities) and denoted $WF_s(u_0)$ for $s \in \ro$.
We show how the singularities of the initial datum $u_0$ are propagated to the solution $u(t,\cdot)$ at time $t>0$, possibly with a change of parameter $s$.
The Gabor wave front set, introduced by H\"ormander \cite{Hormander1}, measures the phase space directions in which a tempered distribution does not behave like a Schwartz function. A tempered distribution has thus empty Gabor wave front set if and only if it is a Schwartz function. In this way the Gabor wave front set is a measure of the lack of 
global regularity, in the sense of both smoothness and decay at infinity. 

To be more precise the Gabor wave front set $WF(u)$ of $u \in \cS'(\rr d)$ is defined negatively as follows. 
For a point in phase space $z \in T^* \rr d \setminus 0$, $z \in T^* \rr d \setminus WF(u)$ means that the short-time Fourier transform of $u$ decays like $C_s \eabs{\cdot}^{-s}$, $C_s>0$, for each $s \in \ro$, in an open conic subset of $T^* \rr d \setminus 0$ that contains $z$. 
The $s$-Gabor wave front set that we introduce and study in this paper is 
a parametrized version of the Gabor wave front set: 
The polynomial decay for each order $s$ in the definition of the Gabor wave front set is relaxed to polynomial decay up to a fixed order $s \in \ro$. Thus $WF_s(u) \subseteq WF(u)$ for all $s \in \ro$ and $u \in \cS'(\rr d)$.  

For the propagation of the $s$-Gabor wave front set
the case of a purely imaginary-valued quadratic form $q$ serves as a reference. 
Denote the solution operator (propagator) of the Schr\"odinger equation by 
$e^{-t q^w(x,D)}$. 
It is then a straightforward consequence of the corresponding result for the Gabor wave front set (see e.g. \cite{Cordero5}) that we have exact propagation of the $s$-Gabor wave front set as  
\begin{equation*}
WF_s(e^{-t q^w(x,D)} u_0) = e^{-2 i t F} WF_s (u_0), \quad t \in \ro, \quad u_0 \in \cS'(\rr d), \quad s \in \ro, 
\end{equation*}
where 
\begin{equation}\label{Jdef}
\J =
\left(
\begin{array}{cc}
0 & I \\
-I & 0
\end{array}
\right) \in \rr {2d \times 2d}, 
\end{equation}
$Q$ is the symmetric matrix that defines $q$, and $F=\J Q$ is the Hamiltonian matrix corresponding to $q$.
This exact propagation is explained by the metaplectic representation. 
Note that the time parameter $t$ is real-valued. The propagator $e^{-t q^w(x,D)}$ is in fact a group. 

If $q$ has a non-negative real part then time has a direction:
The propagation is studied for $t \geqslant 0$ instead of $t \in \ro$ and 
the propagator $e^{-t q^w(x,D)}$ is a semigroup. 
The above equality for the propagation of the $s$-Gabor singularities is then replaced by an inclusion. 

The results of this paper assume that $q$ has a non-negative real part.  
We prove inclusions for $WF_r(e^{- t q^w(x,D)} u_0)$ in terms of $WF_s(u_0)$ and $F$ for $t \geqslant 0$, with 
$r=s$ or $r$ smaller than $s$, either $r < s - 4d$ or $r < s - 8d$. 
A main result is (see Corollary \ref{propagationsing2})
\begin{align}\label{Gaborpropagation}
WF_r (e^{-t q^w(x,D)}u) 
& \subseteq  \left( e^{2 t \im F} \left( WF_s (u) \cap S \right) \right) \cap S, \quad t > 0, 
\end{align}
where $r < s-8d$, $\re F = 2^{-1}(F + \overline{F})$, $\im F = (2i)^{-1}(F- \overline{F})$ and 
\begin{equation*}
S=\Big(\bigcap_{j=0}^{2d-1} \Ker\big[\re F(\im F)^j \big]\Big) \cap T^*\rr d \subseteq T^*\rr d 
\end{equation*}
is the singular space of the quadratic form $q$. 
This result is a refinement of \cite[Theorem 5.2]{Rodino2} that concerns the Gabor wave front set. 
The singular space plays a decisive role for the spectral and hypoelliptic analysis of non-elliptic quadratic operators, cf. \cite{Hitrik1,Hitrik2,Hitrik3,Pravda-Starov1,Pravda-Starov2,Viola1,Viola2}. 

If we assume the Poisson bracket condition $\{q, \overline q \} \equiv 0$ then $S = \Ker (\re F)$. 
Then the inclusion 
\begin{equation*}
WF_s(e^{- t q^w(x,D)} u) \subseteq  
\left( e^{2 t \, \im F} \, \left( WF_s(u) \cap \Ker (\re F) \right) \right) \cap \Ker (\re F), \quad t > 0,
\end{equation*}
holds for $s \in \ro$ (see Proposition \ref{propagationsing3}).
Note that there is no loss of regularity $s$ in this case.  

For the heat equation 
\begin{equation*}
\partial_t u(t,x) - \Delta_x u(t,x) = 0, \quad t \geqslant 0, \quad x \in \rr d, 
\end{equation*}
Proposition \ref{propagationsing3} implies
\begin{equation}\label{Gaborpropheat}
WF_s (e^{-t q^w(x,D)} u_0) 
\subseteq  WF_s (u_0) \cap ( \rr d \times \{ 0 \} ), \quad t > 0, \quad u_0 \in \cS'(\rr d), 
\end{equation}
for $s \in \ro$. 
Therefore, if $s \in \ro$, $u_0 \in \cS'(\rr d)$ and
\begin{equation}\label{WFsassumption0}
WF_s(u_0) \cap ( \rr d \times \{ 0 \} ) = \emptyset 
\end{equation}
then 
\begin{equation}\label{WFsconclusion0}
WF_s (e^{-i t q^w(x,D)} u_0) 
= \emptyset, \quad t > 0.  
\end{equation} 
The regularity assumption in a conic neighborhood of the phase space directions $(\rr d \setminus \{0\}) \times \{0\}$ on the initial datum \eqref{WFsassumption0} gives the global phase space isotropic conclusion \eqref{WFsconclusion0} on the solution $e^{-t q^w(x,D)} u$ for all $t>0$. 

Similarly to the corresponding result for the Gabor wave front set (see \cite[Section 6.2]{Rodino2}), there is an immediate regularizing effect of the heat propagator provided the initial datum has some regularity in the  directions $(\rr d \setminus \{0\}) \times \{0\}$. 

Note that for $s>0$
\begin{equation*}
(\rr d \setminus 0 ) \times \{ 0 \} = WF_s( 1 ). 
\end{equation*}
The heat equation is thus immediately regularizing if $WF_s(u_0)$ is disjoint from $WF_s(1)$ for $s>0$. 

Our results on the propagation of the $s$-Gabor wave front set are refinements of the corresponding results for the Gabor wave front set derived in \cite{Rodino2}. The results in the current paper rely crucially on \cite[Theorem~4.3]{Rodino2} which is an inclusion for the Gabor wave front set of the Schwartz kernel of the propagator. Another important ingredient is Theorem \ref{WFphaseincl} (see Section \ref{secpropagation}) which is an $s$-Gabor front set version of H\"ormander's inclusion \cite[Proposition~2.11]{Hormander1} concerning the Gabor wave front set. Our inclusion reads
\begin{equation}\label{WFpropbasic}
WF_r (\cK u) \subseteq WF(K)' \circ WF_s(u)
\end{equation}
for a linear operator $\cK: \cS (\rr d) \mapsto \cS (\rr d)$ with Schwartz kernel $K \in \cS' (\rr {2d})$, expressed with the twisted Gabor wave front set $WF(K)'$ of the Schwartz kernel, where the fourth coordinate is reflected. 
The Schwartz kernel $K \in \cS' (\rr {2d})$ is assumed to satisfy some conditions regarding its Gabor wave front set, and has a parameter $m \in \ro$. 
The parameters $s,r \in \ro$ are restricted as $r < s - m - 4 d$. 
Since \cite[Theorem~4.3]{Rodino2} gives control of $WF(K)$ when $\cK=e^{-t q^w(x,D)}$ we get propagation of $s$-Gabor singularities results from \eqref{WFpropbasic}. 

The recent papers \cite{Cordero5,Cordero6,Nicola2} treat propagation of Gabor-type wave front sets for various classes of Schr\"odinger equations. These classes admit the presence of potentials with certain properties and symbols $q$ of greater generality than quadratic forms, but they do not admit the quadratic form $q$ to have nonzero real part. 

Concerning the organization of the paper, Section \ref{prelim} fixes notation, introduces background results and defines the $s$-Gabor wave front set. 
In Section \ref{formulation} we specify the Schr\"odinger equation and shortly discuss the semigroup theory underlying the construction of the solution operator when $u_0 \in L^2(\rr d)$. 
In Section \ref{secmicrolocality} some properties of the $s$-Gabor wave front set are proved: independence of the window function, its relation to the Gabor wave front set, microlocality with respect to pseudodifferential operators with symbol in the H\"ormander class $S_{0,0}^0$ and symplectic invariance. Some examples of $s$-Gabor wave front sets of basic tempered distributions are also given. 

Section \ref{secpropagation} is devoted mostly to the proof of the propagation result Theorem \ref{WFphaseincl}. Combined with 
\cite[Theorem~4.3]{Rodino2} we get the propagation of singularities results Corollaries \ref{propagationsing1}, \ref{propagationsing2} and Proposition \ref{propagationsing3}. 

Section \ref{secequations} concerns some examples of concrete equations and how they propagate the $s$-Gabor wave front set according to our results. 

\section{Preliminaries}\label{prelim}

The gradient of a function $f$ with respect to the variable $x \in \rr d$ is denoted by $f'_x$ and
the mixed Hessian matrix with respect to $x \in \rr d$ and $y \in \rr n$ is denoted $f_{x y}''$. 
The Fourier transform of $f \in \mathscr S(\rr d)$ (the Schwartz space) is normalized as
\begin{equation*}
\mathscr{F} f(\xi) = \wh f(\xi) = \int_{\rr d} f(x) e^{- i \la x, \xi \ra} dx,
\end{equation*}
where $\la x, \xi \ra$ denotes the inner product on $\rr d$. The topological dual of $\mathscr S(\rr d)$ is the space of tempered distributions $\mathscr S'(\rr d)$.

The Japanese bracket is $\eabs{x} = (1+|x|^2)^{1/2}$ for $x \in \rr d$. 
For $s \in \ro$ the weighted $L^1$ space $L_s^1(\rr d)$ has norm $f \mapsto \| f \eabs{\cdot}^s \|_{L^1(\rr d)}$. 
An open ball in $\rr d$ of radius $\ep>0$ is denoted $B_\ep (\rr d)= \{ x \in \rr d: \, |x| < \ep \}$.
The unit sphere in $\rr d$ is denoted $S_{d-1} = \{ x \in \rr d: \ |x|=1 \}$. 
For a matrix $A \in \rr {d \times d}$, $A \geqslant 0$ means that $A$ is positive semidefinite, and $A^t$ is the transpose. 
In estimates we denote by $f (x) \lesssim g (x)$ that $f(x) \leqslant C g(x)$ holds for some constant $C>0$ and all $x$ in the domain of $f$ and $g$. 
If $f(x) \lesssim g(x) \lesssim f(x)$ then we write $f(x) \asymp g(x)$. 

We denote the translation operator by $T_x f(y)=f(y-x)$, the modulation operator by $M_\xi f(y)=e^{i \la y, \xi \ra} f(y)$, $x,y,\xi \in \rr d$, and the
phase space translation operator by $\Pi(z) = M_\xi T_x$, $z=(x,\xi) \in \rr {2d}$. 
Given a window function $\varphi \in \mathscr S(\rr d) \setminus \{ 0 \}$, the short-time Fourier transform (STFT) (cf. \cite{Grochenig1}) of $f \in \mathscr S'(\rr d)$ is defined by
\begin{equation}\nonumber
V_\varphi f(z) = ( f, \Pi(z) \varphi ), \quad z \in \rr {2d}.
\end{equation}
Here $(\cdot,\cdot)$ denotes the conjugate linear action of $\cS'$ on $\cS$ which is
consistent with the inner product $(\cdot,\cdot)_{L^2}$ that is conjugate linear in the second argument.
The function $z \mapsto V_\varphi f(z)$ is smooth and its modulus is bounded by $C \eabs{z}^k$ for all $z \in \rr {2d}$ for some $C,k \geqslant 0$.
If $\varphi \in \cS(\rr d)$, $\| \varphi \|_{L^2}=1$ and $f \in \cS'(\rr d)$,
the STFT inversion formula reads (cf. \cite[Corollary~11.2.7]{Grochenig1})
\begin{equation}\label{STFTrecon}
(f,g ) = (2 \pi)^{-d} \int_{\rr {2d}} V_\varphi f(z) ( \Pi(z) \varphi,g ) \, dz, \quad g \in \cS(\rr d).
\end{equation}

Let $\varphi \in \mathscr S(\rr d) \setminus 0$ and let $0 < w \in L_{\rm loc}^\infty (\rr {2d})$ be a weight function that satisfies 
\begin{equation*}
|w(x,\xi)| \lesssim \eabs{(x,\xi)}^s, \quad (x,\xi) \in \rr {2d}, 
\end{equation*}
for some $s \geqslant 0$. 
For $p,q \in [1,\infty]$ the weighted modulation space (cf. \cite{Grochenig1}) $M_{w}^{p,q}(\rr d)$ is the space of $f \in \cS'(\rr d)$ such that the norm
\begin{equation*}
\| f \|_{M_{w}^{p,q}} = \left( \int_{\rr d} \left( \int_{\rr d} \left| V_\varphi f(x,\xi) \, w(x,\xi) \right|^p \, dx \right)^{q/p} \, d \xi \right)^{1/q}
\end{equation*}
is finite, with natural modifications when $p=\infty$ or $q=\infty$. The modulation spaces are Banach spaces and were introduced by Feichtinger \cite{Feichtinger1}. 

The Weyl quantization (cf. \cite{Folland1,Hormander0,Shubin1}) is the map from symbols to operators defined by
\begin{equation}\nonumber
a^w(x,D) f(x) = (2 \pi)^{-d} \iint_{\rr {2d}} e^{i \la x-y,\xi \ra} a \left( \frac{x+y}{2},\xi \right)  \, f(y) \, dy \, d \xi
\end{equation}
for $a \in \mathscr S(\rr {2d})$ and $f \in \mathscr S(\rr d)$. The latter conditions can be relaxed in various ways.
By the Schwartz kernel theorem, any continuous linear operator $\cS(\rr d) \mapsto \cS'(\rr d)$
can be written as a Weyl quantization for a unique symbol $a \in \cS'(\rr {2d})$.

\begin{defn}\label{shubinclasses1}\cite{Shubin1}
For $m\in \ro$, the Shubin symbol class $G^m$ is the subspace of all
$a \in C^\infty(\rr {2d})$ such that for every
$\alpha,\beta \in \nn d$ 
\begin{equation*}
|\partial_x^\alpha \partial_\xi^\beta a(x,\xi)| 
\lesssim \langle (x,\xi) \rangle^{m-|\alpha|-|\beta|}, \quad (x,\xi)\in \rr {2d}. 
\end{equation*}
\end{defn}
We denote $G^\infty = \cup_{m \in \ro} G^m$. 

\begin{defn}\label{hormanderclasses}\cite{Hormander0}
For $m\in \ro$, $0 \leqslant \rho \leqslant 1$, $0 \leqslant \delta < 1$, the H\"ormander symbol class $S_{\rho,\delta}^m$ is the subspace of all
$a \in C^\infty(\rr {2d})$ such that for every
$\alpha,\beta \in \nn d$ 
\begin{equation*}
|\partial_x^\alpha \partial_\xi^\beta a(x,\xi)| 
\lesssim \eabs{\xi}^{m - \rho|\beta| + \delta |\alpha|}, \quad (x,\xi)\in \rr {2d}.  
\end{equation*}
\end{defn}

Both the Shubin symbols and the H\"ormander symbols are Fr\'echet spaces with respect to naturally defined seminorms. 

The following definition introduces conic sets in the phase space $T^* \rr d \simeq \rr {2d}$. 
A set is conic if it is invariant under multiplication with positive reals. 
The distinction to the frequency-conic sets that are used in the definition of the (classical) $C^\infty$ wave front set \cite{Hormander0} should be noted. 

\begin{defn}
Given $a \in G^m$, a point in the phase space $z_0 \in T^* \rr d \setminus 0$ is called non-characteristic for $a$ provided there exist $A,\ep>0$ and an open conic set $\Gamma \subseteq T^* \rr d \setminus 0$ such that $z_0 \in \Gamma$ and
\begin{equation*}
|a(z )| \geqslant \ep \eabs{z}^m, \quad z \in \Gamma, \quad |z| \geqslant A.
\end{equation*}
\end{defn}

The characteristic set $\charac(a)$ of $a \in G^m$ is the complement in $T^* \rr d \setminus 0$ of the set of non-characteristic points. 
The Gabor wave front set is defined as follows. 

\begin{defn}
\cite{Hormander1}
If $u \in \mathscr S'(\rr d)$ then the Gabor wave front set is 
\begin{equation*}
WF(u) = \bigcap_{a \in G^\infty: \ a^w(x,D) u \in \cS} \charac(a) \subseteq T^* \rr d \setminus 0. 
\end{equation*}
\end{defn}

According to \cite[Proposition~6.8]{Hormander1} and \cite[Corollary~4.3]{Rodino1}, the Gabor wave front set can be characterized microlocally by means of the STFT as follows. 
If $u \in \mathscr S'(\rr d)$ and $\varphi \in \mathscr S(\rr d) \setminus 0$ then $z_0 \in T^*\rr d \setminus 0$ satisfies $z_0 \notin WF(u)$ if and only if there exists an open conic set $\Gamma_{z_0} \subseteq T^*\rr d \setminus 0$ containing $z_0$ such that
\begin{equation}\label{WFchar}
\sup_{z \in \Gamma_{z_0}} \eabs{z}^N |V_\varphi u(z)| < \infty \quad \forall N \geqslant 0.
\end{equation}

The microsupport $\mu \supp (a)$ of $a \in G^m$ \cite{Schulz1} is defined as follows. 
An element $z_0 \in T^* \rr d \setminus 0$ satisfies $z_0 \notin \mu \supp (a)$ if there exists an open cone $\Gamma \subseteq T^* \rr d \setminus 0$ that contains $z_0$ and is such that 
\begin{equation*}
\sup_{z \in \Gamma} \eabs{z}^N |\pd \alpha a (z)| < \infty, \quad \alpha \in \nn {2d}, \quad \forall N \geqslant 0.
\end{equation*}
The following list summarizes the most important properties of the Gabor wave front set.

\begin{enumerate}

\item If $u \in \cS'(\rr d)$ then $WF(u) = \emptyset$ if and only if $u \in \cS (\rr d)$ \cite[Proposition~2.4]{Hormander1}.

\item If $u \in \mathscr S'(\rr d)$ and $a \in G^m$ then (cf. \cite{Hormander1,Schulz1,Rodino1})
\begin{align*}
WF( a^w(x,D) u) 
& \subseteq WF(u) \cap \mu \supp (a) \\
& \subseteq WF( a^w(x,D) u) \ \bigcup \ \charac (a). 
\end{align*}

\item If $a \in S_{0,0}^0$ and $u \in \cS'(\rr d)$ then by \cite[Theorem~5.1]{Rodino1}
\begin{equation}\label{microlocal2}
WF(a^w(x,D) u) \subseteq WF(u).
\end{equation}
In particular $WF(\Pi(z) u) = WF(u)$ for any $z \in \rr {2d}$.  

\end{enumerate}

As three basic examples of the Gabor wave front set we have (cf. \cite[Example 6.4--6.6]{Rodino1})
\begin{equation}\label{example1}
WF(\delta_x) = \{ 0 \} \times (\rr d \setminus 0), \quad x \in \rr d,  
\end{equation}
\begin{equation}\label{example2}
WF(e^{i \la \cdot,\xi \ra}) = (\rr d \setminus 0) \times  \{ 0 \}, \quad \xi \in \rr d, 
\end{equation}
and 
\begin{equation}\label{example3}
WF(e^{i \la x, A x \ra/2 } ) = \{ (x, Ax): \, x \in \rr d \setminus 0 \}, \quad A \in \rr {d \times d} \quad \mbox{symmetric}. 
\end{equation}

The canonical symplectic form on $T^* \rr d$ is
\begin{equation*}
\sigma((x,\xi), (x',\xi')) = \la x' , \xi \ra - \la x, \xi' \ra, \quad (x,\xi), (x',\xi') \in T^* \rr d.
\end{equation*}
With the matrix \eqref{Jdef} the symplectic form can be expressed as
\begin{equation*}
\sigma((x,\xi), (x',\xi')) = \la \J (x,\xi), (x',\xi') \ra, \quad (x,\xi), (x',\xi') \in T^* \rr d. 
\end{equation*}

To each symplectic matrix $\chi \in \Sp(d,\ro)$ is associated a unitary operator $\mu(\chi)$ on $L^2(\rr d)$, determined up to a complex factor of modulus one, such that
\begin{equation}\label{symplecticoperator}
\mu(\chi)^{-1} a^w(x,D) \, \mu(\chi) = (a \circ \chi)^w(x,D), \quad a \in \cS'(\rr {2d})
\end{equation}
(cf. \cite{Folland1,Hormander0}).
The operator $\mu(\chi)$ is a homeomorphism on $\mathscr S$ and on $\mathscr S'$.

The mapping $\Sp(d,\ro) \ni \chi \mapsto \mu(\chi)$ is called the \emph{metaplectic representation} \cite{Folland1,Taylor1}.
More precisely it is a representation of the so called $2$-fold covering group of $\Sp(d,\ro)$, which is called the metaplectic group 
and denoted $\Mp(d,\ro)$.
The metaplectic representation satisfies the homomorphism relation modulo a change of sign:
\begin{equation*}
\mu( \chi \chi') = \pm \mu(\chi ) \mu(\chi' ), \quad \chi, \chi' \in \Sp(d,\ro).
\end{equation*}

According to \cite[Proposition~2.2]{Hormander1} the Gabor wave front set is symplectically invariant as
\begin{equation}\label{symplecticinvarianceWF}
WF( \mu(\chi) u) = \chi WF(u), \quad \chi \in \Sp(d, \ro), \quad u \in \cS'(\rr d).
\end{equation}

For $s \in \ro$ we introduce the $s$-Gabor wave front set $WF_s (u)$ of $u \in \cS'(\rr d)$. 
This concept is a refinement of the Gabor wave front set $WF(u)$ in the sense that $WF_s(u) \subseteq WF(u)$ for all $s \in \ro$.
The $s$-Gabor wave front set is the notion of singularity we study is this paper.

\begin{defn}\label{wavefront2}
Let $\varphi \in \mathscr S(\rr d) \setminus 0$, $u \in \mathscr S'(\rr d)$ and $s \in \ro$. 
Then $z_0 \in T^*\rr d \setminus 0$ satisfies $z_0 \notin WF_s (u)$ if there exists an open conic set $\Gamma_{z_0} \subseteq T^*\rr d \setminus 0$ containing $z_0$ such that
\begin{equation*}
\sup_{z \in \Gamma_{z_0}} \eabs{z}^s |V_\varphi u(z)| < \infty.
\end{equation*}
\end{defn}

We will show in Proposition \ref{sGaborinvariance} that the definition of $WF_s (u)$ does not depend on the window function $\fy \in \mathscr S(\rr d) \setminus 0$. 

It follows that $WF_s (u) = \emptyset$ if and only if $u \in \cS'(\rr d)$ satisfies
\begin{equation*}
|V_\varphi u(z)| \lesssim \eabs{z}^{-s}, \quad z \in \rr {2d}.  
\end{equation*}
For any $u \in \cS'(\rr d)$ the STFT $V_\varphi u$ is polynomially bounded so $WF_s (u) = \emptyset$ provided $s$ is sufficiently small for each $u \in \cS'(\rr d)$. 

The $s$-Gabor wave front set $WF_s (u)$ increases with the index $s$: 
\begin{equation*}
t \geqslant s \quad \Longrightarrow \quad WF_s(u) \subseteq WF_t(u). 
\end{equation*}

From $WF_s(u) \subseteq WF(u)$ and \eqref{example1} we have for any $s \in \ro$
\begin{equation*}
WF_s(\delta_0) \subseteq \{ 0 \} \times (\rr d \setminus 0 ).  
\end{equation*}
If $\fy \in \cS \setminus 0$ satisfies $\fy(0) \neq 0$ then for $\xi \in \rr d \setminus 0$ and $t>0$
\begin{equation*}
|V_\fy \delta_0 (0, t \xi)| = |\fy(0)| \neq 0
\end{equation*}
so 
\begin{equation*}
\{ 0 \} \times (\rr d \setminus 0 ) \subseteq WF_s(\delta_0), \quad s>0.
\end{equation*}
Hence 
\begin{equation}\label{example1sa0}
WF_s(\delta_0) = \{ 0 \} \times (\rr d \setminus 0 ), \quad s>0, 
\end{equation}
Since 
\begin{equation*}
|V_\fy \delta_0 (x, \xi)| = |\fy(-x)| \lesssim 1, \quad (x,\xi) \in \rr {2d}, 
\end{equation*}
we have 
\begin{equation}\label{example1sb0}
WF_s(\delta_0) = \emptyset, \quad s \leqslant 0.  
\end{equation}

\section{Problem formulation and the solution operator}
\label{formulation}

We study the initial value Cauchy problem for a Schr\"odinger equation of the form
\begin{equation}\label{schrodeq}
\left\{
\begin{array}{rl}
\partial_t u(t,x) + q^w(x,D) u (t,x) & = 0, \\
u(0,\cdot) & = u_0, 
\end{array}
\right.
\end{equation}
where $t \geqslant 0$, $x \in \rr d$ and $u_0 \in \cS'(\rr d)$.
The Hamiltonian $q^w(x,D)$ has a Weyl symbol that is a quadratic form
\begin{equation*}
q(x,\xi) = \la (x, \xi), Q (x, \xi) \ra, \quad x, \, \xi \in \rr d, 
\end{equation*}
where $Q \in \cc {2d \times 2d}$ is a symmetric matrix with $\re Q \geqslant 0$. 
The special case $\re Q = 0$ admits to study the equation for $t \in \ro$ instead of $t \geqslant 0$. 

The \emph{Hamilton map} $F$ corresponding to $q$ is defined by 
\begin{equation*}
\sigma(Y, F X) = q(Y,X), \quad X,Y \in \rr {2d}, 
\end{equation*}
where $q(Y,X)$ is the bilinear polarized version of the form $q$, i.e. $q(X,Y)=q(Y,X)$ and $q(X,X)=q(X)$. 
The Hamilton map $F$ is thus the matrix 
\begin{equation*}
F = \J Q \in \cc {2d \times 2d}
\end{equation*}
with $\J$ defined by \eqref{Jdef}. 

When $u_0 \in L^2(\rr d)$ the equation \eqref{schrodeq} is solved for $t \geqslant 0$ by 
\begin{equation*}
u(t,x) = e^{-t q^w(x,D)} u_0(x)
\end{equation*}
where the solution operator (propagator) $e^{-t q^w(x,D)}$ is the contraction semigroup that is generated by the operator $-q^w(x,D)$. 
Contraction semigroup means a strongly continuous semigroup with $L^2$ operator norm $\leqslant 1$ for all $t \geqslant 0$ \cite{Yosida1}. 
The reason why $- q^w(x,D)$, or more precisely its closure $M_{-q}$ as an unbounded linear operator in $L^2(\rr d)$, generates such a semigroup is explained in \cite[pp.~425--26]{Hormander2}. The contraction semigroup property is a consequence of $M_{- q}$ and its adjoint $M_{-\overline q}$ being \emph{dissipative} operators \cite{Yosida1}, which for $M_{-q}$ means that
\begin{equation*}
\re (M_{-q} u,u) = (M_{-\re q} u,u) \leqslant 0, \quad u \in D(M_{- q}), 
\end{equation*}
$D(M_{-q}) \subseteq L^2(\rr d)$ denoting the domain of $M_{-q}$. 
The fact that $M_{- q}$ and $M_{-\overline q}$ are dissipative follows from the assumption $\re Q \geqslant 0$. 

Our objective in this work is the propagation of the $s$-Gabor wave front set for the Schr\"odinger propagator $e^{-t q^w(x,D)}$. 
This means that we want to find  inclusions for 
\begin{equation*}
WF_r(e^{-t q^w(x,D)} u_0)
\end{equation*}
in terms of $WF_s(u_0)$, $F$ and $t \geqslant 0$ for $u_0 \in \cS'(\rr d)$ and as sharp conditions on $r,s \in \ro$ as possible. 

If $\re Q=0$ then the propagator is given by means of the metaplectic representation. 
To wit, if 
$\re Q=0$ then $e^{-t q^w(x,D)}$ is a group of unitary operators, and we have by \cite[Theorem~4.45]{Folland1}
\begin{equation}\label{propagatorsymplectic}
e^{-t q^w(x,D)} = \mu(e^{-2 i t F}), \quad t \in \ro,  
\end{equation}
where $\mu$ is the metaplectic representation, see  \eqref{symplecticoperator}. 
In this case $F$ is purely imaginary and $i F \in \ssp(d,\ro)$, the symplectic Lie algebra, which implies that $e^{-2 i t F} \in \Sp(d,\ro)$ for any $t \in \ro$ \cite{Folland1}.
Since $\mu(e^{-2 i t F})$ is a continuous operator on $\cS'$, the semigroup theory extends uniquely to initial datum $u_0 \in \cS'(\rr d)$ instead of merely $u_0 \in L^2(\rr d)$. 
According to  
\eqref{symplecticinvarianceWF} we have
\begin{equation}\label{realcase}
WF(e^{-t q^w(x,D)} u_0) = e^{-2 i t F} WF(u_0), \quad t \in \ro, \quad u_0 \in \cS'(\rr d). 
\end{equation}

\section{Properties of the $s$-Gabor wave front set}\label{secmicrolocality}

\begin{lem}\label{convolutioninvariance}
Let $f$ be a measurable function that satisfies for $M \geqslant 0$
\begin{equation}\label{polynomialbound1}
|f(x)| \lesssim \eabs{x}^{M}, \quad x \in \rr d.
\end{equation}
Let $s \in \ro$
and suppose there exists a non-empty open conic set
$\Gamma \subseteq \rr d \setminus 0$ 
such that
\begin{equation}\label{conedecay1}
\sup_{x \in \Gamma} \eabs{x}^s |f(x)| < \infty.
\end{equation}
If
\begin{equation}\label{L1intersection}
g \in \bigcap_{t \geqslant 0} L_{t}^1(\rr d)
\end{equation}
then for any open conic set $\Gamma' \subseteq \rr d \setminus 0$ 
such that
$\overline{\Gamma' \cap S_{d-1}} \subseteq \Gamma$, we have
\begin{equation}\label{conedecay2}
\sup_{x \in \Gamma'} \eabs{x}^s |f * g(x)| < \infty.
\end{equation}
\end{lem}

\begin{proof}
By the assumptions \eqref{polynomialbound1} and \eqref{L1intersection} and Peetre's inequality 
\begin{equation*}
\eabs{x+y}^t \lesssim \eabs{x}^t \eabs{y}^{|t|}, \quad x,y \in \rr d, \quad t \in \ro, 
\end{equation*}
we have
\begin{equation*}
|f * g (x)| \lesssim \eabs{x}^M, \quad x \in \rr d, 
\end{equation*}
so it suffices to assume $|x| \geqslant 1$. 

Let $\ep>0$.
We estimate and split the convolution integral as
\begin{equation*}
|f * g(x)|
\leqslant  \underbrace{\int_{\eabs{y} \leqslant \ep \eabs{x}} |f(x-y)| \, | g (y)| \, d y}_{:= I_1}
+ \underbrace{\int_{\eabs{y} > \ep \eabs{x}} |f(x-y)| \, | g (y)| \, d y}_{:= I_2}.
\end{equation*}
Consider $I_1$.
We may assume that $\Gamma'$ is non-empty. 
Since $\eabs{y} \leqslant \ep \eabs{x}$ we have $x-y \in \Gamma$ if $x \in \Gamma'$, $|x| \geqslant 1$, and $\ep>0$ is chosen sufficiently small.
The assumptions \eqref{conedecay1} and \eqref{L1intersection} give
\begin{equation}\label{intuppsk1}
\begin{aligned}
I_1 & \lesssim \int_{\eabs{y} \leqslant \ep \eabs{x}} \eabs{x-y}^{-s} |g (y)| \, d y
\lesssim \eabs{x}^{-s} \int_{\rr d} \eabs{y}^{|s|} |g(y)| \, d y \\
& \lesssim  \eabs{x}^{-s}, \quad x \in \Gamma', \quad |x| \geqslant 1.
\end{aligned}
\end{equation}
Next we estimate $I_2$ using \eqref{polynomialbound1}. 
If necessary we first increase $M$ so that $M \geqslant -s$. 
This gives
\begin{equation}\label{intuppsk2}
\begin{aligned}
I_2 & \lesssim \int_{\eabs{y} > \ep \eabs{x}} \eabs{x-y}^{M} \, |g(y) | \, dy \\
& \lesssim \eabs{x}^{M} \int_{\eabs{y} > \ep \eabs{x}}  \eabs{y}^{-M-s} \, \eabs{y}^{2M+s} \, |g(y) | \, dy \\
& \lesssim \eabs{x}^{M-M-s} \int_{\rr d} \eabs{y}^{2M+s} \, |g(y) | \, dy \\
& \lesssim \eabs{x}^{-s}, \quad x \in \rr d,
\end{aligned}
\end{equation}
again using \eqref{L1intersection}.
A combination of \eqref{intuppsk1} and \eqref{intuppsk2} proves \eqref{conedecay2}.
\end{proof}

As a first application of Lemma \ref{convolutioninvariance} we show that  Definition \ref{wavefront2} does not depend on the Schwartz function $\fy \in \cS(\rr d) \setminus 0$. 

\begin{prop}\label{sGaborinvariance}
Suppose $u \in \cS'(\rr d)$. The definition of the $s$-Gabor wave front set $WF_s(u)$ does not depend on the window function $\fy \in \cS(\rr d) \setminus 0$. 
\end{prop}

\begin{proof}
Let $\fy,\psi \in \cS(\rr d) \setminus 0$. 
By \cite[Theorem~11.2.3]{Grochenig1} we have for some $M \geqslant 0$
\begin{equation*}
|V_\varphi u (z)| \lesssim \eabs{z}^{M}, \quad z \in \rr {2d}, 
\end{equation*}
and by \cite[Lemma~11.3.3]{Grochenig1} we have 
\begin{equation*}
|V_\psi u(z)| \leqslant (2 \pi)^{-d} \| \fy \|_{L^2} |V_\fy u| * |V_\psi \fy| (z), \quad z \in \rr {2d}. 
\end{equation*}
If $|V_\fy u(z)|$ decays like $\eabs{z}^{-s}$ in a conic set $\Gamma \subseteq T^* \rr d \setminus 0$ containing $z_0 \neq 0$ then by Lemma \ref{convolutioninvariance} we get decay for $|V_\psi u(z)|$ of order $\eabs{z}^{-s}$ in a smaller cone containing $z_0$, since 
\begin{equation*}
V_\psi \fy \in \cS(\rr {2d}) \subseteq \bigcap_{t \geqslant 0} L_{t}^1(\rr {2d}). 
\end{equation*} 
Hence, by symmetry, polynomial decay of order $s \in \ro$ in an open cone around a point in $T^* \rr d \setminus 0$ happens simultaneously for $V_\fy u$ and $V_\psi u$. 
\end{proof}

For a conical subset $\Gamma \subseteq T^* \rr d \setminus 0$ we denote by $\overline \Gamma \subseteq T^* \rr d \setminus 0$ its closure with respect to the usual topology in $T^* \rr d \simeq \rr {2d}$ in $T^* \rr d \setminus 0$. We have the following equality:

\begin{prop}\label{WFuWFsu}
If $u \in \cS'(\rr d)$ then 
\begin{equation*}
WF(u) = \overline{\bigcup_{s \in \ro} WF_s(u)}.
\end{equation*} 
\end{prop}

\begin{proof}
The inclusion
\begin{equation*}
WF(u) \supseteq \overline{\bigcup_{s \in \ro} WF_s(u)}
\end{equation*} 
follows immediately by virtue of
$WF_s(u) \subseteq WF(u)$ for all $s \in \ro$
and $WF(u) \subseteq T^* \rr d \setminus 0$ being closed. 

To show 
\begin{equation}\label{WFsubWFs}
WF(u) \subseteq \overline{\bigcup_{s \in \ro} WF_s(u)}
\end{equation} 
we may assume that $\overline{\bigcup_{s \in \ro} WF_s(u)} \neq T^* \rr d \setminus 0$. 

Let $0 \neq z_0 \notin \overline{\bigcup_{s \in \ro} WF_s(u)}$.
There exists an open conic set $\Gamma \subseteq T^* \rr d \setminus 0$ containing $z_0$ such that 
$\overline \Gamma \cap \overline{\bigcup_{s \in \ro} WF_{s}(u)} = \emptyset$. 

Let $\fy \in \cS(\rr d) \setminus 0$. 
By definition of $WF_{s}(u)$, for every $z \in \overline \Gamma$ and every $s \in \ro$ there exists an open conical set $\Gamma_{z,s} \subseteq T^* \rr d \setminus 0$ containing $z$ such that 
\begin{equation*}
\sup_{w \in \Gamma_{z,s}} \eabs{w}^{s} | V_\varphi u (w)| < \infty.
\end{equation*}
We thus have for each $s \in \ro$
\begin{equation*}
\overline{ \Gamma \cap S_{2d-1}} \subseteq  \bigcup_{z \in \overline{\Gamma \cap S_{2d-1}}} \Gamma_{z,s}
\end{equation*}
and by the compactness of $\overline{ \Gamma \cap S_{2d-1}} \subseteq T^* \rr d \setminus 0$, the covering of open sets on the right hand side may be reduced to a finite covering of the form 
\begin{equation}\label{gammacovering}
\overline{ \Gamma \cap S_{2d-1}} \subseteq  \bigcup_{j=1}^n \Gamma_{z_j,s} := \Gamma_s, \quad s \in \ro, 
\end{equation}
where $\Gamma_s \subseteq T^*\rr d \setminus 0$ is open, conic and satisfies  
\begin{equation*}
\sup_{w \in \Gamma_{s}} \eabs{w}^{s} | V_\varphi u (w)| <  \infty, \quad s \in \ro. 
\end{equation*}
From \eqref{gammacovering} we may conclude
\begin{equation*}
\Gamma \subseteq  \bigcap_{s \in \ro} \Gamma_s
\end{equation*}
which gives
\begin{equation*}
\sup_{w \in \Gamma} \eabs{w}^{N} | V_\varphi u (w)|
\leqslant \sup_{w \in \Gamma_N} \eabs{w}^{N} | V_\varphi u (w)| < \infty, \quad N \geqslant 0. 
\end{equation*}
We may conclude that $z_0 \notin WF(u)$ which proves the inclusion \eqref{WFsubWFs}.
\end{proof}

The next result says that pseudodifferential operators with symbols in $S_{0,0}^0$ are microlocal with respect to the $s$-Gabor wave front set. 
The proof is similar to the corresponding proof for the Gabor wave front set \cite[Theorem~5.1]{Rodino1}. 

\begin{prop}\label{microlocal1}
If $s \in \ro$, $a \in S_{0,0}^0$ and $u \in \cS'(\rr d)$ then
\begin{equation*}
WF_s( a^w(x,D) u) \subseteq WF_s(u).
\end{equation*}
\end{prop}

\begin{proof}
We have (see e.g. \cite{Holst1})
\begin{equation}\label{symbolintersection}
S_{0,0}^0 = \bigcap_{N \geqslant 0} M_{v_N}^{\infty,1}
\end{equation}
where $v_N$ is the weight $v_N(x,\xi) = \eabs{\xi}^N$ for $(x,\xi) \in \rr {2d} \oplus \rr {2d}$, and $M_{v_N}^{\infty,1} = M_{v_N}^{\infty,1}(\rr {2d})$
is a weighted modulation space.
The space $M_{v_N}^{\infty,1}$
is also known as a weighted version of Sj\"ostrand's symbol class (cf. \cite{Grochenig2,Sjostrand1,Sjostrand2}).

Let $\varphi \in \cS(\rr d)$ satisfy $\| \varphi \|_{L^2}=1$.
Denoting the formal adjoint of $a^w(x,D)$ by $a^w(x,D)^*$, \eqref{STFTrecon} gives for $z \in \rr {2d}$
\begin{align*}
V_\varphi (a^w(x,D) u) (z)
& = ( a^w(x,D) u, \Pi(z) \varphi ) \\
& = ( u, a^w(x,D)^* \Pi(z) \varphi ) \\
& = (2 \pi)^{-d} \int_{\rr {2d}} V_\varphi u(w) \, ( \Pi(w) \varphi,a^w(x,D)^* \Pi(z) \varphi ) \, dw \\
& = (2 \pi)^{-d} \int_{\rr {2d}} V_\varphi u(w) \, ( a^w(x,D) \, \Pi(w) \varphi,\Pi(z) \varphi ) \, dw \\
& = (2 \pi)^{-d} \int_{\rr {2d}} V_\varphi u(z-w) \, ( a^w(x,D) \, \Pi(z-w) \varphi,\Pi(z) \varphi ) \, dw.
\end{align*}
By \eqref{symbolintersection} and \cite[Theorem~3.2]{Grochenig2}, for any $t \geqslant 0$ there exists $g_t \in L_{t}^1(\rr {2d})$ such that
\begin{align*}
\left| ( a^w(x,D) \, \Pi(z-w) \varphi,\Pi(z) \varphi ) \right| \leqslant g_t(w), \quad z, w \in \rr {2d}.
\end{align*}
With $g(w) = \sup_{z \in \rr {2d}} | ( a^w(x,D) \, \Pi(z-w) \varphi,\Pi(z) \varphi ) |$
we thus have
\begin{equation*}
g \in \bigcap_{t \geqslant 0} L_t^1(\rr {2d}),
\end{equation*}
and
\begin{align}\label{convolution1}
|V_\varphi (a^w(x,D) u) (z)|
& \lesssim |V_\varphi u| * g(z), \quad z \in \rr {2d}.
\end{align}
If $0 \neq z_0 \in T^* \rr d \setminus WF_s(u)$ then there exists an open conic set $\Gamma \subseteq T^* \rr d \setminus 0$ containing $z_0$ such that
\begin{equation*}
\sup_{z \in \Gamma} \eabs{z}^s |V_\varphi u(z)| < \infty.
\end{equation*}
By \cite[Theorem~11.2.3]{Grochenig1} we have for some $M \geqslant 0$
\begin{equation*}
|V_\varphi u (z)| \lesssim \eabs{z}^{M}, \quad z \in \rr {2d}.
\end{equation*}
It now follows from \eqref{convolution1} and Lemma \ref{convolutioninvariance}
that for any open conic set $\Gamma'$ containing $z_0$ such that $\overline{\Gamma' \cap S_{2d-1}} \subseteq \Gamma$ we have
\begin{equation*}
\sup_{z \in \Gamma'} \eabs{z}^s |V_\varphi (a^w(x,D) u) (z)| < \infty,
\end{equation*}
which proves that $z_0 \notin WF_s( a^w(x,D) u)$.
Thus we have shown
\begin{equation*}
WF_s( a^w(x,D) u) \subseteq WF_s(u).
\end{equation*}
\end{proof}

\begin{rem}
Note that Proposition \ref{microlocal1} is a refinement of \eqref{microlocal2} (cf. \cite[Theorem~5.1]{Rodino1}), in view of Proposition \ref{WFuWFsu}.
\end{rem}

Since modulation and translation are invertible operators with Weyl symbols in $S_{0,0}^0$,
the result gives the following consequence.

\begin{cor}
If $u \in \cS'(\rr d)$ and $z \in \rr {2d}$ then
\begin{equation*}
WF_s( \Pi(z) u ) = WF_s(u).
\end{equation*}
\end{cor}

Thus the examples \eqref{example1sa0} and \eqref{example1sb0} generalizes into  
\begin{equation}\label{example1sa}
WF_s(\delta_x) = \{ 0 \} \times (\rr d \setminus 0 ), \quad x \in \rr d, \quad s>0, 
\end{equation}
and
\begin{equation}\label{example1sb}
WF_s(\delta_x) = \emptyset, \quad x \in \rr d, \quad s \leqslant 0, 
\end{equation}
respectively. 

Next we show an $s$-Gabor wave front set version of the symplectic invariance \eqref{symplecticinvarianceWF}. 

\begin{lem}\label{symplecticGabors}
For each $s \in \ro$
\begin{equation}\label{symplecticinvarianceWFs}
WF_s ( \mu(\chi) u) = \chi WF_s (u), \quad \chi \in \Sp(d, \ro), \quad u \in \cS'(\rr d), 
\end{equation}
where $\mu(\chi)$ is the operator corresponding to $\chi \in \Sp(d, \ro)$ that satisfies \eqref{symplecticoperator}. 
\end{lem}

\begin{proof}
Let $\varphi \in \mathscr S(\rr d) \setminus 0$. Since $\mu(\chi) \varphi \in \cS(\rr d) \setminus 0$ it suffices by Definition \ref{wavefront2} to show
\begin{equation*}
\left| V_{ \mu(\chi) \varphi} ( \mu(\chi) u )( \chi z)\right| = \left| V_\varphi u ( z)\right|, \quad z \in \rr {2d}.
\end{equation*}

Define for $x,\xi \in \rr d$ the symbol
\begin{equation*}
a_{x,\xi} (y,\eta) = e^{i \la x,\xi \ra/2 + i (\la \xi,y \ra - \la x,\eta \ra)}, \quad y,\eta \in \rr d.
\end{equation*}
For $f,g \in \mathscr S$ we have $( a_{x,\xi}^w(x,D) f,g ) = (M_\xi T_x f,g)$, that is $a_z^w(x,D) = \Pi(z)$ for $z \in \rr {2d}$.
Note that $a_{x,\xi} \in S_{0,0}^0$.
It follows from \eqref{symplecticoperator} that
\begin{equation*}
\mu(\chi)^{-1} \, \Pi(\chi z) \, \mu (\chi) = (a_{\chi z} \circ \chi)^w(x,D).
\end{equation*}
By \cite[Proposition~4.1]{Folland1} we have, with $A,B,C,D \in \rr {d \times d}$,
\begin{equation*}
\chi =
\left(
  \begin{array}{ll}
  A & B \\
  C & D
  \end{array}
\right) \in \Sp(d,\ro)
\end{equation*}
if and only if
\begin{equation*}
A^t C = C^t A, \quad B^t D = D^t B \quad \mbox{and} \quad A^t D - C^t B = I.
\end{equation*}
With $z=(x,\xi)$ this gives 
\begin{align*}
a_{\chi z} \circ \chi(y,\eta)
& = e^{i \la A x + B \xi, C x + D \xi \ra/2 + i (\la C x + D \xi,A y + B \eta \ra - \la A x + B \xi,C y + D \eta \ra)} \\
& = e^{i( \la A x, C x \ra + \la B \xi, D \xi \ra + \la x, (A^t D+C^t B ) \xi \ra)/2 } \\
& \quad \times e^{i (\la x, (C^t B - A^t D) \eta \ra + \la \xi , (D^t A - B^t C) y \ra)} \\
& = e^{i ( \la A x, C x \ra + \la B \xi, D \xi \ra + 2 \la Cx, B \xi \ra + \la x, \xi \ra) /2} \, e^{i (\la \xi,y \ra - \la x,\eta \ra)} \\
& = e^{i ( \la A x, C x \ra + \la B \xi, D \xi \ra + 2 \la Cx, B \xi \ra) /2} a_{x,\xi}(y,\eta),
\end{align*}
and hence $(a_{\chi z} \circ \chi)^w(x,D) = e^{i ( \la A x, C x \ra + \la B \xi, D \xi \ra + 2 \la Cx, B \xi \ra ) /2} \Pi(x,\xi)$.
This gives finally for $z \in \rr {2d}$
\begin{align*}
\left| V_{\mu(\chi) \varphi} (\mu(\chi) u )( \chi z)\right|
& = \left| ( \mu(\chi) u, \Pi(\chi z) \, \mu(\chi) \varphi ) \right|
= \left| ( u, \mu(\chi) ^{-1} \Pi(\chi z) \, \mu(\chi) \varphi ) \right| \\
& = \left| ( u, \Pi(z) \varphi ) \right|
= \left| V_\varphi u (z) \right|.
\end{align*}
\end{proof}

Combining \eqref{propagatorsymplectic} and \eqref{symplecticinvarianceWFs} we obtain the following invariance result for the $s$-Gabor wave front set (cf. \eqref{realcase}). 
If $Q \in \cc {2d \times 2d}$, $\re Q = 0$ and $F = \J Q  \in \rr {2d \times 2d}$ then 
\begin{equation}\label{realcaseWFs}
WF_s(e^{- t q^w(x,D)} u_0) = e^{-2 i t F} WF_s(u_0), \quad t \in \ro, \quad u_0 \in \cS'(\rr d), \quad s \in \ro. 
\end{equation}

The symplectic invariance \eqref{symplecticinvarianceWFs} gives $s$-Gabor wave front set versions of \eqref{example1}, \eqref{example2} and \eqref{example3} as follows.  
We have \eqref{example1sa}, \eqref{example1sb}, 
\begin{equation*}
WF_s(e^{i \la \cdot,\xi \ra}) = WF_s (e^{i \la \cdot, A \, \cdot \ra/2 } ) = \emptyset, \quad \xi \in \rr d, \quad A \in \rr {d \times d} \quad \mbox{symmetric}, \quad s \leqslant 0,
\end{equation*} 
\begin{equation}\label{example2sa}
WF_s(e^{i \la \cdot,\xi \ra}) = (\rr d \setminus 0 ) \times  \{ 0 \}, \quad \xi \in \rr d, \quad s>0, 
\end{equation}
and 
\begin{equation}\label{example3sa}
WF_s (e^{i \la \cdot, A \, \cdot \ra/2 } ) = \{ (x, Ax): \, x \in \rr d \setminus 0 \}, \quad A \in \rr {d \times d} \quad \mbox{symmetric}, \quad s > 0.
\end{equation}

\section{Propagation of polynomial phase space singularities}\label{secpropagation}

\subsection{Propagation of $s$-Gabor singularities for certain linear operators}

Every continuous linear operator $\cK: \cS(\rr d) \mapsto \cS'(\rr d)$ has a Schwartz kernel $K \in \cS'(\rr {2d})$ that satisfies
\begin{equation*}
(\cK f, g) = (K, g \otimes \overline f), \quad f,g \in \cS(\rr d). 
\end{equation*}

\begin{lem}\label{STFTop}
Let $\cK: \cS(\rr d) \mapsto \cS'(\rr d)$ be a continuous linear operator with Schwartz kernel $K \in \cS'(\rr {2d})$, 
suppose $\fy \in \cS(\rr d)$ satisfies $\| \fy \|_{L^2} = 1$ and set $\Phi = \fy \otimes \fy$. 
Then for $u, \psi \in \cS(\rr d)$
\begin{equation}\label{TSTFT}
\begin{aligned}
(\cK u, \psi) 
= (2 \pi)^{-2d} \int_{\rr {4d}} V_\Phi K(x,y,\xi,-\eta) \, \overline{V_\fy \psi (x,\xi)} \, V_{\overline \fy} u(y,\eta) \, dx \, dy \, d \xi \, d \eta. 
\end{aligned}
\end{equation}
\end{lem}

\begin{proof}
Since $M_{\xi,\eta} T_{x,y} \Phi = M_\xi T_x \fy \otimes M_\eta T_y \fy$ we have 
\begin{equation*}
\begin{aligned}
V_\Phi (\psi \otimes \overline u) (x,y,\xi,\eta) 
& = (\psi \otimes \overline u, M_\xi T_x \fy \otimes M_\eta T_y \fy) \\
& = V_\fy \psi (x,\xi) \, \overline{V_{\overline \fy} u(y,-\eta) }. 
\end{aligned}
\end{equation*}
The STFT inversion formula \eqref{STFTrecon} gives
\begin{equation*}
\begin{aligned}
(\cK u, \psi) 
&  = (K, \psi \otimes \overline u) \\
& = (2 \pi)^{-2d} \int_{\rr {4d}} V_\Phi K (x,y,\xi,-\eta) \, \overline{V_\fy \psi (x,\xi)} \, V_{\overline \fy} u(y,\eta) \, dx \, dy \, d \xi \, d \eta. 
\end{aligned}
\end{equation*}
The integral converges since $V_\Phi K$ is polynomially bounded according to \cite[Theorem~11.2.3]{Grochenig1}, and $V_\fy \psi, V_{\overline \fy} u \in \cS(\rr {2d})$. 
\end{proof}

Since
\begin{equation*}
\overline{V_\fy \Pi(t,\theta) \fy (x,\xi)} = e^{i \la x, \xi - \theta \ra} V_\fy \fy (t-x,\theta-\xi)
\end{equation*}
we obtain from Lemma \ref{STFTop} with $\psi = \Pi(t,\theta) \fy$ for $(t,\theta) \in \rr {2d}$
\begin{equation}\label{STFTop1}
\begin{aligned}
& V_\fy(\cK u) (t, \theta) 
= (\cK u, \Pi(t,\theta) \fy) \\
& = (2 \pi)^{-2d} \int_{\rr {4d}} e^{i \la x,\xi-\theta \ra} V_\Phi K (x,y,\xi,-\eta) V_\fy \fy (t-x,\theta-\xi) \, V_{\overline \fy} u(y,\eta) \, dx \, dy \, d \xi \, d \eta. 
\end{aligned}
\end{equation}
This formula will be useful in the proof of Theorem \ref{WFphaseincl}. 

In the following results we need some definitions from \cite{Hormander1}. 
For $K \in \cS'(\rr {2d})$ we define
\begin{equation}\label{WFproj}
\begin{aligned}
WF_1(K) & = \{ (x,\xi) \in T^* \rr d: \ (x, 0, \xi, 0) \in WF(K) \} & \subseteq T^* \rr d \setminus 0, \\
WF_2(K) & = \{ (y,\eta) \in T^* \rr d: \ (0, y, 0, -\eta) \in WF(K) \} & \subseteq T^* \rr d \setminus 0. 
\end{aligned}
\end{equation}

We also need the relation mapping between a subset $A \subseteq X \times Y$ of the Cartesian product of two sets $X$, $Y$, and a subset $B \subseteq Y$, 
\begin{equation*}
A \circ B = \{ x \in X: \, \exists y \in B: \, (x,y) \in A \} \subseteq X.  
\end{equation*}
Finally we need later the reflection operator in the fourth $\rr d$ coordinate on $\rr {4d}$
\begin{equation}\label{twist}
(x,y,\xi,\eta)' = (x,y,\xi,-\eta), \quad x,y,\xi,\eta \in \rr d. 
\end{equation}

The following lemma concerns the relation mapping between closed conic sets, the first of which does not intersect the coordinate axes. 

\begin{lem}\label{konlemma1}
Let $G \subseteq \rr {2d} \setminus 0$ and $G_1 \subseteq \rr d \setminus 0$ be closed conic sets such that 
\begin{equation}\label{assumpconelem}
G \cap \left( \{0\} \times \rr d \cup \rr d \times \{0\}  \right) = \emptyset. 
\end{equation}
Suppose $x_0 \in \rr d \setminus 0$
satisfies $x_0 \notin G \circ G_1$. 
Then there exist open conic sets $\Gamma_0, \Gamma_1 \subseteq \rr d \setminus 0$ such that $x_0 \in \Gamma_0$, $G_1 \subseteq \Gamma_1$ and 
$\overline{\Gamma}_0 \cap (G \circ \overline{\Gamma}_1) = \emptyset$.
\end{lem}

\begin{proof}
The assumption $x_0 \notin G \circ G_1$ means 
\begin{equation}\label{asskonlem}
(\{ x_0 \} \times G_1) \cap G = \emptyset. 
\end{equation}
Let $0 < \ep < 1$ and define the open conic sets
\begin{align*}
\Gamma_{0,\ep} & = \left\{ x \in \rr d \setminus 0 : \,  \frac{x}{|x|} \in \frac{x_0}{|x_0|} + B_\ep(\rr d) \right\} \subseteq \rr d \setminus 0, \\
\Gamma_{1,\ep} & = \left\{ y \in \rr d \setminus 0 : \,  \frac{y}{|y|} \in G_1 + B_\ep(\rr d) \right\} \subseteq \rr d \setminus 0, 
\end{align*}
whose closures in $\rr d \setminus 0$ are, respectively,
\begin{align*}
\overline{\Gamma}_{0,\ep} & = \left\{ x \in \rr d \setminus 0 : \,  \frac{x}{|x|} \in \frac{x_0}{|x_0|} + \overline{B_\ep(\rr d)} \right\} \subseteq \rr d \setminus 0, \\
\overline{\Gamma}_{1,\ep} & = \left\{ y \in \rr d \setminus 0 : \,  \frac{y}{|y|} \in G_1 + \overline{B_\ep(\rr d)} \right\} \subseteq \rr d \setminus 0. 
\end{align*}
Obviously $x_0 \in \Gamma_{0,\ep}$ and $G_1 \subseteq \Gamma_{1,\ep}$ for any $\ep>0$. 

We are going to show that 
\begin{equation}\label{conclkonlem}
( \overline{\Gamma}_{0,\ep} \times \overline{\Gamma}_{1,\ep}) \cap G = \emptyset
\end{equation}
holds for some $\ep>0$, which proves the lemma since it is equivalent to $\overline{\Gamma}_{0,\ep} \cap (G \circ \overline{\Gamma}_{1,\ep}) = \emptyset$. 

To prove \eqref{conclkonlem} suppose for a contradiction that 
\begin{equation}\label{conecontradiction}
( \overline{\Gamma}_{0,1/n} \times \overline{\Gamma}_{1,1/n}) \cap G \neq \emptyset, \quad n \in \no. 
\end{equation}
Then there exists 
$(x_n,y_n) \in ( \overline{\Gamma}_{0,1/n} \times \overline{\Gamma}_{1,1/n}) \cap G$  
for all $n \in \no$. 
Since all involved sets are conic we may assume
\begin{equation*}
(x_n,y_n) \in ( \overline{\Gamma}_{0,1/n} \times \overline{\Gamma}_{1,1/n}) \cap G  \cap S_{2d-1}, \quad n \in \no.
\end{equation*}
Passing to a subsequence without change of notation gives convergence
\begin{equation*}
(x_n,y_n) \longrightarrow (x,y) \in 
G \cap S_{2d-1}, \quad n \rightarrow \infty,
\end{equation*}
and thanks to the assumption \eqref{assumpconelem} we must have $x \neq 0$ and $y \neq 0$. 

Since $x_n \in \overline{\Gamma}_{0,1/n}$ we have 
\begin{equation*}
 \frac{x_n}{|x_n|} =  \frac{x_0}{|x_0|} + w_n
\end{equation*}
where $|w_n| \leqslant 1/n$, so it follows that $x \in \ro_+ x_0 = \{t x_0: \, t>0 \}$. 
Likewise, since $y_n \in \overline{\Gamma}_{1,1/n}$ we have 
\begin{equation*}
 \frac{y_n}{|y_n|} =  u_n + w_n
\end{equation*}
where $u_n \in G_1$ and $|w_n| \leqslant 1/n$. 
It follows from the closure of $G_1$ in $\rr d \setminus 0$ and its cone property that $y \in G_1$. 

We have thus deduced
\begin{equation*}
(x,y) \in (\ro_+ x_0 \times G_1) \cap G 
\end{equation*}
which contradicts \eqref{asskonlem}. 
It follows that our assumption \eqref{conecontradiction} must be false so \eqref{conclkonlem} indeed holds for some $\ep>0$. 
\end{proof}

For $K \in \cS'(\rr {2d})$ and $\Phi \in \cS(\rr {2d}) \setminus 0$ we have by \cite[Theorem~11.2.3]{Grochenig1} (cf. Section \ref{prelim}) 
for some $m \in \ro$ that does not depend on $\Phi$, 
\begin{equation}\label{kernelSTFT}
|V_\Phi K(x,y,\xi,\eta)| \lesssim \eabs{(x,y,\xi,\eta)}^m, \quad (x,y,\xi,\eta) \in \rr {4d}. 
\end{equation}

The preceding two lemmas are needed in the following result on propagation of singularities.
It is an $s$-Gabor wave front set version of H\"ormander's result \cite[Proposition~2.11]{Hormander1} which treats the Gabor wave front set. 
More precisely H\"ormander's result concerns a continuous linear operator $\cK : \cS(\rr d) \mapsto \cS'(\rr d)$ with Schwartz kernel $K \in \cS'(\rr {2d})$. 
The proposition says that the domain of $\cK $ can be extended to all $u \in \cS'(\rr d)$ such that 
\begin{equation*}
WF (u) \cap WF_2(K) = \emptyset,
\end{equation*}
in which case $\cK u \in \cS'(\rr d)$, 
and the Gabor wave front set inclusion
\begin{equation*}
WF (\cK u) \subseteq WF(K)' \circ WF (u) \cup WF_1(K)
\end{equation*}
holds. 
Note that $WF(K)' \circ WF (u)$ here means 
\begin{align*}
& WF(K)' \circ WF (u) \\
& = \{ (x,\xi) \in T^* \rr d: \,  \exists (y,\eta) \in WF (u) : \, (x,y,\xi,-\eta) \in WF(K) \}, 
\end{align*}
that is, the second and third $\rr d$ coordinates of $WF(K)$ are permuted (see \cite{Hormander1,Hormander2,Rodino2}).

\begin{thm}\label{WFphaseincl}
Suppose $\cK : \cS(\rr d) \mapsto \cS(\rr d)$ is a continuous linear operator that extends uniquely to a continuous linear operator $\cK : \cS' (\rr d) \mapsto \cS' (\rr d)$.
Suppose the Schwartz kernel $K \in \cS'(\rr {2d})$ of $\cK $ satisfies \eqref{kernelSTFT} for $m \in \ro$, and 
\begin{equation}\label{WKjempty}
WF_1(K) = WF_2(K) = \emptyset.
\end{equation}  
Then for $s,r \in \ro$ such that 
\begin{equation}\label{parameterineq}
r < s - m - 4d
\end{equation}
and $u \in \cS'(\rr d)$ we have
\begin{equation}\label{WFphaseinclusion}
WF_r (\cK u) \subseteq WF(K)' \circ WF_s (u).  
\end{equation}
\end{thm}

\begin{proof}
By \cite[Lemma~5.1]{Carypis1} we have by the assumption \eqref{WKjempty}
\begin{equation*}
WF(K) \subseteq \Gamma_1 = \{(x,y,\xi,\eta) \in T^* \rr {2d}: \ c^{-1} |(x,\xi)| <  |(y,\eta)| < c |(x,\xi)| \} 
\end{equation*}
for some $c > 1$. 
Defining 
\begin{align*}
\Gamma_{1,3} & = \{(x,y,\xi,\eta) \in T^* \rr {2d}: \ c |(x,\xi)| \leqslant |(y,\eta)| \}, \\
\Gamma_{2,4} & = \{(x,y,\xi,\eta) \in T^* \rr {2d}: \ c |(y,\eta)| \leqslant  |(x,\xi)| \}, \\
\end{align*}
we thus have
\begin{equation}\label{inclusionG1}
\Gamma_1 \subseteq \rr {4d} \setminus (\Gamma_{1,3} \cup \Gamma_{2,4} ). 
\end{equation}

After these preparations we first show that the formula \eqref{TSTFT} extends to $u \in \cS' (\rr d)$ and $\psi \in \cS (\rr d)$ under the given assumptions. 

By \cite[Corollary~11.2.6]{Grochenig1} 
the topology for $\cS (\rr d)$ can be defined by the collection of seminorms
\begin{equation}\label{seminorm}
\cS(\rr d) \ni \psi \mapsto \sup_{z \in \rr {2d}} \eabs{z}^n |V_\fy \psi (z)|, \quad n \geqslant 0,
\end{equation}
for any $\fy \in \cS(\rr d) \setminus 0$. 
Pick $\fy \in \cS(\rr d)$ such that $\| \fy \|_{L^2} = 1$ and set $\Phi = \fy \otimes \fy$. 
By Lemma \ref{STFTop} we have for $\psi,u \in \cS(\rr d)$
\begin{equation}\label{seminormest0}
|(\cK u,\psi)| 
\lesssim \int_{\rr {4d}} |V_\Phi K(x,y,\xi,-\eta)| \,  |V_\fy \psi (x,\xi)| \, |V_{\overline \fy} u(y,\eta)| \, dx \, dy \, d \xi \, d \eta. 
\end{equation}
We first show that the right hand side integral can be estimated by a seminorm \eqref{seminorm} of $\psi$ when $u \in \cS'(\rr d)$. 
By the assumed uniqueness of the extension $\cK : \cS' (\rr d) \mapsto \cS' (\rr d)$ this implies that 
formula \eqref{TSTFT} holds for $u \in \cS' (\rr d)$ and $\psi \in \cS (\rr d)$.

Consider first the right hand side integral over $(x,y,\xi,-\eta) \in \rr {4d} \setminus \Gamma_1$, where we have 
\begin{equation}\label{WFcompl}
|V_\Phi K(x,y,\xi,-\eta)| \lesssim \eabs{(x,y,\xi,\eta)}^{-k}, \quad k \in \no, \quad (x,y,\xi,-\eta) \in \rr {4d} \setminus \Gamma_1, 
\end{equation}
on account of $WF(K) \subseteq \Gamma_1$ and $\Gamma_1 \subseteq \rr {4d} \setminus 0$ being open. 
By \cite[Theorem~11.2.3]{Grochenig1} we have for some $L \geqslant 0$
\begin{equation}\label{STFTu}
|V_{\overline \fy} u(y,\eta)| \lesssim \eabs{(y,\eta)}^{L}, \quad (y,\eta) \in \rr {2d}, 
\end{equation}
since $u \in \cS'(\rr d)$.  

We have for any $n \in \no$
\begin{equation}\label{seminormest1}
\begin{aligned}
& \int_{\rr {4d} \setminus \Gamma_1'} 
|V_\Phi K(x,y,\xi,-\eta)| \,  |V_\fy \psi (x,\xi)| \, |V_{\overline \fy} u(y,\eta)| \, dx \, dy \, d \xi \, d \eta \\
& \lesssim 
\int_{\rr {4d} \setminus \Gamma_1'} 
\eabs{(x,y,\xi,\eta)}^{-k} \,  \eabs{(x,\xi)}^n |V_\fy \psi (x,\xi)| \, \eabs{(y,\eta)}^{L} \, dx \, dy \, d \xi \, d \eta \\
& \lesssim \sup_{z \in \rr {2d}} \eabs{z}^n |V_\fy \psi (z)| 
\int_{\rr {4d}}  \eabs{(x,y,\xi,\eta)}^{L-k} \, dx \, dy \, d \xi \, d \eta \\
& \lesssim \sup_{z \in \rr {2d}} \eabs{z}^n |V_\fy \psi (z)|
\end{aligned}
\end{equation}
provided $k > 0$ is sufficiently large. 

Next we consider the right hand side integral \eqref{seminormest0} over $(x,y,\xi,-\eta) \in \Gamma_1$. 
If $(x,y,\xi,-\eta) \in \Gamma_1$ we have $\eabs{(x,\xi)} \asymp \eabs{(y,\eta)}$ 
which gives
\begin{equation}\label{seminormest2}
\begin{aligned}
& \int_{\Gamma_1'} 
|V_\Phi K(x,y,\xi,-\eta)| \,  |V_\fy \psi (x,\xi)| \, |V_{\overline \fy} u(y,\eta)| \, dx \, dy \, d \xi \, d \eta \\
& \lesssim 
\int_{\Gamma_1'} 
\eabs{(x,y,\xi,\eta)}^{m+4d+1-4d-1} \, |V_\fy \psi (x,\xi)| \, \eabs{(y,\eta)}^{L} \, dx \, dy \, d \xi \, d \eta \\
& \lesssim 
\int_{\Gamma_1'} 
\eabs{(x,y,\xi,\eta)}^{-4d-1} \, \eabs{(x,\xi)}^{|m|+4d+1+L} \, |V_\fy \psi (x,\xi)| \, dx \, dy \, d \xi \, d \eta \\
& \lesssim \sup_{z \in \rr {2d}} \eabs{z}^{|m| + 4 d + 1 + L} |V_\fy \psi (z)|. 
\end{aligned}
\end{equation}
The estimates \eqref{seminormest1} and \eqref{seminormest2} prove our claim that  
the right hand side of \eqref{seminormest0} can be estimated by a seminorm \eqref{seminorm} of $\psi$ when $u \in \cS'(\rr d)$. 
Thus \eqref{TSTFT}, and therefore also \eqref{STFTop1}, hold for $u \in \cS' (\rr d)$ and $\psi \in \cS (\rr d)$.

Using \eqref{STFTop1} we show the inclusion \eqref{WFphaseinclusion} under assumption \eqref{parameterineq} by showing that 
\begin{equation}\label{assumption1}
0 \neq (t_0,\theta_0) \notin WF(K)' \circ WF_s (u) 
\end{equation}
implies $(t_0,\theta_0) \notin WF_r (\cK u)$.
Thus we suppose \eqref{assumption1}. 
By Lemma \ref{konlemma1} we may assume that $(t_0,\theta_0) \in \Omega_0$ and $\overline \Omega_0 \cap (WF(K)' \circ \overline \Omega_2) = \emptyset$
where $\Omega_0, \Omega_2 \subseteq T^* \rr d \setminus 0$ are conic, open and  $WF_s (u) \subseteq \Omega_2$. 
(Note that the assumption \eqref{assumpconelem} of Lemma \ref{konlemma1} corresponds to the assumption \eqref{WKjempty}.)

Denote by 
\begin{align*}
p_{1,3}(x,y,\xi,\eta) & = (x,\xi), \\
p_{2,-4}(x,y,\xi,\eta) & = (y,-\eta), \quad x,y,\xi, \eta \in \rr d, 
\end{align*}
the projections $\rr {4d} \mapsto \rr {2d}$ onto the first and the third $\rr d$ coordinate, 
and onto the second and the fourth $\rr d$ coordinate with a change of sign in the latter, respectively. 

With these notations we may express $\overline \Omega_0 \cap (WF(K)' \circ \overline \Omega_2) = \emptyset$ as
\begin{equation*}
\overline \Omega_0 \cap p_{1,3} \left( WF(K) \cap p_{2,-4}^{-1} \, \overline \Omega_2 \right) = \emptyset, 
\end{equation*}
or, equivalently, 
\begin{equation*}
p_{1,3}^{-1} \, \overline{\Omega}_0 \cap WF(K) \cap p_{2,-4}^{-1} \, \overline \Omega_2 = \emptyset.  
\end{equation*}
Due to assumption \eqref{WKjempty} we may strengthen this into 
\begin{equation*}
p_{1,3}^{-1} \, (\overline{\Omega}_0 \cup \{ 0 \} ) \setminus 0 \cap WF(K) \cap p_{2,-4}^{-1} \, (\overline \Omega_2 \cup \{ 0 \} ) \setminus 0 = \emptyset.  
\end{equation*}

Since $p_{1,3}^{-1} \, (\overline{\Omega}_0 \cup \{ 0 \} ) \setminus 0$ and $p_{2,-4}^{-1} \, (\overline \Omega_2 \cup \{ 0 \} ) \setminus 0$ are closed conic subsets of $\rr {4d} \setminus 0$, 
decreasing $\Gamma_1 \subseteq \rr {4d} \setminus 0$ if necessary
there exist open conic subsets $\Gamma_0, \Gamma_2 \subseteq \rr {4d} \setminus 0$ such that 
\begin{equation*}
WF(K) \subseteq \Gamma_1, \qquad
p_{1,3}^{-1} \, \overline{\Omega}_0 \subseteq \Gamma_0, \qquad
p_{2,-4}^{-1} \, \overline{\Omega}_2 \subseteq \Gamma_2, 
\end{equation*}
and
\begin{equation}\label{intersection1}
\Gamma_0 \cap \Gamma_1 \cap \Gamma_2 = \emptyset. 
\end{equation}

Let $\Sigma_0 \subseteq T^* \rr d \setminus 0$ be an open conic set such that 
$(t_0,\theta_0) \in \Sigma_0$ and $\overline{\Sigma_0 \cap S_{2d-1}} \subseteq \Omega_0$. 
Suppose $\fy \in \cS(\rr d)$, $\| \fy \|_{L^2} = 1$ and $\Phi = \fy \otimes \fy$.
From \eqref{STFTop1}  we have
\begin{equation}\label{estimand1}
\begin{aligned}
& \eabs{(t,\theta)}^{r}  |V_\fy(\cK u) (t, \theta)| \\
& \lesssim
\int_{\rr {4d}} |V_\Phi K(x,y,\xi,-\eta)| \, \eabs{(t,\theta)}^{r} \, | V_\fy \fy (t-x,\theta-\xi)| 
\, |V_{\overline \fy} u(y,\eta)| \, dx \, dy \, d \xi \, d \eta. 
\end{aligned}
\end{equation}
We will show that this integral is bounded when $(t,\theta) \in \Sigma_0$ which proves that $(t_0,\theta_0) \notin WF_r (\cK u)$. 

Consider first the right hand side integral over $(x,y,\xi,-\eta) \in \rr {4d} \setminus \Gamma_1$. 
From \eqref{WFcompl}, \eqref{STFTu} and $\eabs{(t,\theta)} \lesssim \eabs{(x,\xi)} \eabs{(t-x,\theta-\xi)}$ we obtain, since $V_\fy \fy \in \cS(\rr {2d})$, 
\begin{equation}\label{estimateA}
\begin{aligned}
& \int_{\rr {4d} \setminus \Gamma_1'} |V_\Phi K(x,y,\xi,-\eta)| \, \eabs{(t,\theta)}^{r}  \, | V_\fy \fy (t-x,\theta-\xi)| 
\, |V_{\overline \fy} u(y,\eta)| \, dx \, dy \, d \xi \, d \eta \\
& \lesssim \int_{\rr {4d} \setminus \Gamma_1'} \eabs{(x,y,\xi,\eta)}^{-k} \eabs{(x,\xi)}^{r} \eabs{(y,\eta)}^{L} \, dx \, dy \, d \xi \, d \eta \\
& \lesssim \int_{\rr {4d}} \eabs{(x,y,\xi,\eta)}^{|r| + L - k} \, dx \, dy \, d \xi \, d \eta
< \infty
\end{aligned}
\end{equation}
if $k > 0$ is chosen sufficiently large. The estimate holds for all $(t,\theta) \in \rr {2d}$.  

It remains to estimate the right hand side integral \eqref{estimand1} over $(x,y,\xi,-\eta) \in \Gamma_1$ 
where $\eabs{(x,\xi)} \asymp \eabs{(y,\eta)}$.  
By \eqref{inclusionG1} and \eqref{intersection1} we have $\Gamma_1 \subseteq G_1 \cup G_2$ where
\begin{equation*}
G_1 = \rr {4d} \setminus (\Gamma_{1,3} \cup \Gamma_{2,4} \cup \Gamma_0), \quad G_2 = \rr {4d} \setminus (\Gamma_{1,3} \cup \Gamma_{2,4} \cup \Gamma_2 ). 
\end{equation*}
First we assume $(x,y,\xi,-\eta) \in G_1$. Then $(x,y,\xi,-\eta) \notin \Gamma_0$ which implies $(x,\xi) \notin \Omega_0$. 
There exists $\delta>0$ such that
\begin{equation*}
|(x,\xi) - (t,\theta)| \geqslant \delta |(x,\xi)|, \quad (x,\xi) \notin \Omega_0, \quad (t,\theta) \in \Sigma_0. 
\end{equation*}

For $(t,\theta) \in \Sigma_0$ we obtain with the aid of \eqref{kernelSTFT}, for arbitrary $k \geqslant |r|$, since $V_\fy \fy \in \cS(\rr {2d})$, 
using $\eabs{(x,\xi)} \asymp \eabs{(y,\eta)}$ and \eqref{STFTu}, 
\begin{equation}\label{estimateB}
\begin{aligned}
& \int_{G_1'} |V_\Phi K(x,y,\xi,-\eta)| \, \eabs{(t,\theta)}^{r} \,  | V_\fy \fy (t-x,\theta-\xi)| \, 
|V_{\overline \fy} u(y,\eta)| \, dx \, dy \, d \xi \, d \eta \\
& \lesssim \int_{G_1'} \eabs{(x,y,\xi,\eta)}^{m} \, \eabs{(x,\xi)}^{r} \,  
\eabs{(t-x,\theta-\xi)}^{|r|-k} \, \eabs{(t-x,\theta-\xi)}^{-k} \\
& \qquad \qquad \qquad \qquad \qquad \qquad \qquad \qquad \qquad \qquad \times \eabs{(y,\eta)}^L \, dx \, dy \, d \xi \, d \eta \\
& \lesssim \int_{G_1'} \eabs{(x,\xi)}^{|m|+r + 4 d + 2 + L -k} \, \eabs{(x,\xi)}^{-2d-1} \eabs{(y,\eta)}^{-2d-1} \, dx \, dy \, d \xi \, d \eta \\
& \lesssim \int_{\rr {4d}} \eabs{(x,\xi)}^{-2d-1} \eabs{(y,\eta)}^{-2d-1} \, dx \, dy \, d \xi \, d \eta 
\lesssim 1
\end{aligned}
\end{equation}
provided $k \geqslant |m| + r + 4 d + 2 + L$. 

Finally we assume $(x,y,\xi,-\eta) \in G_2$. Then $(x,y,\xi,-\eta) \notin \Gamma_2$ so we have $(y,\eta) \notin \Omega_2$. 
Hence $(y,\eta) \in G$ where $G \subseteq T^* \rr d$ is closed, conic and does not intersect $WF_s(u)$. 
Assumption \eqref{parameterineq} gives
\begin{equation*}
\ep = s - r - m - 4 d > 0.
\end{equation*}
We obtain with the aid of \eqref{kernelSTFT} for any $(t,\theta) \in T^* \rr d$,
using $\eabs{(x,\xi)} \asymp \eabs{(y,\eta)}$, 
\begin{equation}\label{estimateC}
\begin{aligned}
& \int_{G_2'} |V_\Phi K(x,y,\xi,-\eta)| \, \eabs{(t,\theta)}^{r}  | V_\fy \fy (t-x,\theta-\xi)| 
\, |V_{\overline \fy} u(y,\eta)| \, dx \, dy \, d \xi \, d \eta \\
& \lesssim \int_{G_2'} \eabs{(x,y,\xi,\eta)}^{m} \eabs{(x,\xi)}^{r} |V_{\overline \fy} u(y,\eta)|  \, dx \, dy \, d \xi \, d \eta \\
& \lesssim \int_{G_2'} \eabs{(x,y,\xi,\eta)}^{-4d-\ep} \, \eabs{(y,\eta)}^{r + m + 4 d +\ep} |V_{\overline \fy} u(y,\eta)|  \, dx \, dy \, d \xi \, d \eta \\
& \lesssim \sup_{w \in G} \, \eabs{w}^{s} |V_{\overline \fy} u(w)|
< \infty. 
\end{aligned}
\end{equation}

We can now combine \eqref{estimand1}, \eqref{estimateA}, $\Gamma_1 \subseteq G_1 \cup G_2$, \eqref{estimateB} and \eqref{estimateC} to conclude
\begin{equation*}
\sup_{(t,\theta) \in \Sigma_0} \eabs{(t,\theta)}^r \, |V_\fy (\cK u) (t,\theta)| < \infty. 
\end{equation*}
Thus $(t_0,\theta_0) \notin WF_r (\cK u)$. 
\end{proof}

\begin{rem}
Under assumption \eqref{WKjempty}, Theorem \ref{WFphaseincl} implies H\"ormander's result \cite[Proposition~2.11]{Hormander1} for the Gabor wave front set
\begin{equation*}
WF (\cK u) \subseteq WF(K)' \circ WF (u).
\end{equation*}
In fact this follows from Proposition \ref{WFuWFsu} and the claim $\overline{G \circ G_1} \subseteq G \circ \overline{G_1}$ 
if $G \subseteq \rr {2d} \setminus 0$ is closed, conic and satisfies \eqref{assumpconelem}, and $G_1 \subseteq \rr d \setminus 0$ is merely conic (cf. Lemma \ref{konlemma1}). 

To show the claim, assume that $x_n \in G \circ G_1$ for $n \geqslant 1$ and $x_n \rightarrow x \in \rr d$ as $n \rightarrow \infty$. 
For each $n \in \no$ there exists $y_n \in G_1$  such that $(x_n,y_n) \in G$. 

Suppose that the sequence $\{ y _n \}_{n=1}^\infty \subseteq G_1$ is not bounded. 
Then for some subsequence we have $|y_{n_j}| \geqslant j$ for $j \in \no$.
The cone property of $G$ gives $|y_{n_j}|^{-1} (x_{n_j}, y_{n_j}) \in G$, and this sequence converges to $(0,y) \in G$ for some $y \in \overline{G_1} \cap S_{d-1}$ as $j \rightarrow \infty$, after passage to a further subsequence. This contradicts the assumption \eqref{assumpconelem}, and it follows that the sequence $\{ y _n \}_{n=1}^\infty \subseteq G_1$ must be bounded.

Therefore we have convergence
\begin{equation*}
(x_{n_j}, y_{n_j} ) \rightarrow (x,y) \in G, \quad j \rightarrow \infty, 
\end{equation*}
for some subsequence, and $y \in \overline{G_1} \cup \{ 0 \}$. 
Due to the assumption \eqref{assumpconelem} we have $x \neq 0$ and $y \neq 0$, so $y \in \overline{G_1}$ which gives $x \in G \circ \overline{G_1}$. 
This proves the claim $\overline{G \circ G_1} \subseteq G \circ \overline{G_1}$.
\end{rem}

\subsection{Propagation of $s$-Gabor singularities in the general case}

In this subsection we combine Theorem \ref{WFphaseincl} with results in \cite{Rodino2} in order to obtain results on propagation of singularities for Schr\"odinger equations. 
H\"ormander's result \cite[Theorem~5.12]{Hormander2} says that the Schwartz kernel of the Schr\"odinger propagator $e^{-t q^w(x,D)}$ is a so called Gaussian oscillatory integral. 
Therefore we first need some background on Gaussian oscillatory integrals and the associated symplectic geometry. 
We give a brief description of these topics and refer to \cite[Section~5]{Hormander2} and \cite[Sections~3 and 4]{Rodino2} for richer accounts. 

Let $q$ be a quadratic form on $T^*\rr d$ defined by a symmetric matrix $Q \in \cc {2 d \times 2 d}$ with $\re Q \geqslant 0$ and Hamilton matrix $F=\J Q$.
According to \cite[Theorem~5.12]{Hormander2} the propagator is 
\begin{equation*}
e^{-t q^w(x,D)} = \cK_{e^{-2 i t F}}
\end{equation*}
where $\cK_{e^{-2 i t F}}: \cS(\rr d) \mapsto \cS'(\rr d) $ is the linear continuous operator with Schwartz kernel 
\begin{equation}\label{schwartzkernel1}
K_T (x,y) = (2 \pi)^{-(d+N)/2} \sqrt{\det \left( 
\begin{array}{ll}
p_{\theta \theta}''/i & p_{\theta y}'' \\
p_{x \theta}'' & i p_{x y}'' 
\end{array}
\right) } \int_{\rr N} e^{i p(x,y,\theta)} d\theta \in \cS'(\rr {2d}) 
\end{equation}
with $T = {e^{-2 i t F}}$. 
The Schwartz kernel $K_T$ is a Gaussian oscillatory integral with respect to a quadratic form $p$ on $\rr {2 d + N}$ \cite{Hormander2,Rodino2}. 
The quadratic form $p$ satisfies certain properties (cf. \cite[Section~3]{Rodino2}) including $\im p \geqslant 0$, and it is not uniquely determined by the oscillatory integral. 
Each Gaussian oscillatory integral is bijectively associated with a Lagrangian in $T^* \cc {2d}$. 
Certain aspects of the analysis of the propagator are natural to study in terms of the corresponding Lagrangian. 

The Lagrangian associated to the Schwartz kernel $K_{e^{-2 i t F}}$ is 
\begin{equation}\label{twistgraph}
\lambda = \{ (x, y, \xi, -\eta) \in T^* \cc {2d}: \, (x,\xi) =  e^{-2 i t F} (y,\eta) \} \subseteq T^* \cc {2d}.
\end{equation}
With the twist operation (cf. \eqref{twist})
\begin{equation*}
(x, y, \xi, \eta)' = (x, y, \xi, -\eta), \quad x, y, \xi, \eta \in \cc d, 
\end{equation*}
the Lagrangian is thus the twisted graph Lagrangian defined by the matrix $e^{-2 i t F}$. 
By \cite[Lemma~4.2]{Rodino2} we have $T=e^{-2 i t F} \in \Sp(d,\co)$, and 
\begin{equation*}
i \left(  \sigma(\overline{T X}, TX) - \sigma(\overline X, X) \right) \geqslant 0, \quad X \in T^* \cc d. 
\end{equation*}
A matrix $T \in \Sp(d,\co)$ that satisfies this property is called \emph{positive} \cite[p.~444]{Hormander2}.  

Since $K_{e^{-2 i t F}} \in \cS'(\rr {2d})$ the propagator is a continuous operator $e^{-t q^w(x,D)}: \cS (\rr d) \mapsto \cS' (\rr d)$. 
According to \cite[Proposition~5.8 and Theorem~5.12]{Hormander2} we have in fact continuity $e^{-t q^w(x,D)}: \cS (\rr d) \mapsto \cS (\rr d)$. 
In order to extend the propagator to an operator $e^{-t q^w(x,D)}: \cS' (\rr d) \mapsto \cS' (\rr d)$ one considers the formal adjoint. 
The formal adjoint of $\cK_T$ for a positive matrix $T \in \Sp(d,\co)$ is $\cK_{\overline T^{-1}}$ 
where $\overline T^{-1}  \in \Sp(d,\co)$ is positive \cite[p.~446]{Hormander2}. 
Thus $\cK_T$ may be extended uniquely to a continuous operator $\cK_T: \cS' (\rr d) \mapsto \cS' (\rr d)$ by
\begin{equation*}
( \cK_T u, \fy) = (u, \cK_{\overline T^{-1}} \fy), \quad u \in \cS'(\rr d), \quad \fy \in \cS(\rr d).
\end{equation*}
Since $e^{-t q^w(x,D)} = \cK_{e^{-2 i t F}}$ this gives in particular the unique extension of the propagator to a continuous operator $e^{-t q^w(x,D)}: \cS' (\rr d) \mapsto \cS' (\rr d)$. 

According to \cite[Theorem~4.3]{Rodino2}, the Gabor wave front set of the Schwartz kernel $K_{e^{-2 i t F}}$ of the propagator $e^{-t q^w(x,D)}$ for $t \geqslant 0$ obeys the  inclusion
\begin{equation}\label{WFkernel1}
\begin{aligned}
& WF( K_{e^{-2 i t F}} ) \\
& \subseteq \{ (x, y, \xi, -\eta) \in T^* \rr {2d} \setminus 0 : \, (x,\xi) =  e^{-2 i t F} (y,\eta), \, \im e^{-2 i t F} (y,\eta) = 0 \}. 
\end{aligned}
\end{equation}
Since $e^{-2 i t F} \in \cc {2d \times 2d}$ is invertible, $WF_1 ( K_{e^{-2 i t F}} ) = WF_2 ( K_{e^{-2 i t F}} )= \emptyset$, cf. \eqref{WFproj}.

A Gaussian oscillatory integral with respect to a quadratic form as in \eqref{schwartzkernel1} is a particular case of a so called Gaussian distribution \cite[Section~5]{Hormander2}. 
An oscillatory integral of this type has the form
\begin{equation*}
C e^{i \rho} (\delta_0 \otimes 1) \circ A  
\end{equation*}
where $\rho$ is a quadratic form on $\rr {2d}$ with $\im \rho \geqslant 0$, $\delta_0=\delta_0(\rr k)$, $k \leqslant 2d$, $A \in \rr {2d \times 2d}$ is an invertible matrix 
and $C \in \co \setminus 0$. 
From this information it follows that we may take $m=0$ in \eqref{kernelSTFT} when $K=K_{e^{-2 i t F}}$. 

The observations above allow us to combine \eqref{WFkernel1} and Theorem \ref{WFphaseincl} with $\cK=e^{-t q^w(x,D)}$ and $m=0$. 
This gives the following propagation result for Schr\"odinger equations and the $s$-Gabor wave front set.  

\begin{cor}\label{propagationsing1}
Suppose $q$ is a quadratic form on $T^*\rr d$ defined by a symmetric matrix $Q \in \cc {2d \times 2d}$, $\re Q \geqslant 0$, $F=\J Q$, 
$s,r \in \ro$ and $r < s - 4d$.
Then for $u \in \cS'(\rr d)$
\begin{align*}
WF_r (e^{-t q^w(x,D)}u) 
& \subseteq  e^{-2 i t F} \left( WF_s (u) \cap \Ker (\im e^{-2 i t F} ) \right), \quad t \geqslant 0. 
\end{align*}
\end{cor}

As in the proof of \cite[Theorem~5.2]{Rodino2}, the latter inclusion can be sharpened (at the expense of decreasing $r$) using the semigroup property of the propagator \cite[Proposition~5.9 and Theorem~5.12]{Hormander2}
\begin{equation*}
e^{-(t_1+t_2)q^w(x,D)}=\pm e^{-t_1q^w(x,D)}e^{-t_2q^w(x,D)}, \quad t_1,t_2 \geqslant 0. 
\end{equation*} 
One obtains then the following result, where the singular space 
\begin{equation*}
S=\Big(\bigcap_{j=0}^{2d-1} \Ker\big[\re F(\im F)^j \big]\Big) \cap T^*\rr d \subseteq T^*\rr d 
\end{equation*}
of the quadratic form $q$ plays a crucial role. 
The singular space has recently been found to play a decisive role in the analysis of spectral and hypoelliptic properties of non-elliptic quadratic operators, 
 cf. \cite{Hitrik1,Hitrik2,Hitrik3,Pravda-Starov1,Pravda-Starov2,Viola1,Viola2}. 
 
The following result is an $s$-Gabor wave front set refinement of \cite[Theorem~5.2]{Rodino2} 
(which is formulated in terms of the Gabor wave front set). 
This result implies \cite[Theorem~5.2]{Rodino2} via Proposition \ref{WFuWFsu}. 
Note that the upper bound for $r$ is $4d$ smaller than in Corollary \ref{propagationsing1}. 
As shown in the proof of \cite[Theorem~5.2]{Rodino2}
\begin{equation*}
\left( e^{2 t \im F} \left( WF_s (u) \cap S \right) \right) \cap S \subseteq e^{-2 i t F} \left( WF_s (u) \cap \Ker (\im e^{-2 i t F} ) \right)
\end{equation*}
so the right hand side of the following inclusion is a sharpening of the right hand side in the inclusion of Corollary \ref{propagationsing1}. 

\begin{cor}\label{propagationsing2}
Suppose $q$ is a quadratic form on $T^*\rr d$ defined by a symmetric matrix $Q \in \cc {2d \times 2d}$, $\re Q \geqslant 0$, $F=\J Q$,  
$s,r \in \ro$ and $r < s - 8 d$.
Then for $u \in \cS'(\rr d)$
\begin{align*}
WF_r (e^{-t q^w(x,D)}u) 
& \subseteq  \left( e^{2 t \im F} \left( WF_s (u) \cap S \right) \right) \cap S, \quad t > 0. 
\end{align*}
\end{cor}

\subsection{Consequences of the assumption of normality of the Hamiltonian}

Finally we draw some conclusions on propagation of the $s$-Gabor wave front set under the additional hypothesis of normality of the Hamiltonian. 
This means that $q^w(x,D)$ commutes with its formal adjoint $\overline{q}^w(x,D)$, that is  
\begin{equation*}
[q^w(x,D), \overline{q}^w(x,D)] = 0  
\end{equation*}
acting on $\mathscr{S}(\rr d)$.
This condition of normality is equivalent to the condition on the Poisson bracket 
\begin{equation*}
\{q,\overline{q}\} 
= \la q_{\xi}', \overline q_{x}' \ra - \la q_{x}', \overline q_{\xi}' \ra 
= 2i\{\im q,\re q\} \equiv 0.
\end{equation*}
The condition is also equivalent to $[\re F,\im F]=0$
since the Hamilton map of the Poisson bracket $\{\im q,\re q\}$ is given by $-2[\im F,\re F]$, see e.g. \cite[Lemma~2]{Pravda-Starov1}.
In this case the singular space reduces to 
\begin{equation*}
S=\Ker (\re F)\cap T^*\rr d.
\end{equation*}

The following result is an $s$-Gabor wave front set refinement of \cite[Corollary~5.3]{Rodino2}. 
Again it implies the latter result via Proposition \ref{WFuWFsu}. 
Note that there is no loss of order $s$ as opposed to Corollary \ref{propagationsing2}. 

\begin{prop}\label{propagationsing3}
Suppose $q$ is a quadratic form on $T^*\rr d$ defined by a symmetric matrix $Q \in \cc {2d \times 2d}$ such that $\re Q \geqslant 0$. 
If $F = \J Q$ satisfy $[\re F,\im F] = 0$ then for 
$s \in \ro$ and $t>0$
we have for $u \in \cS'(\rr d)$
\begin{equation*}
\begin{aligned}
WF_s\big(e^{-tq^w(x,D)}u\big) 
\subseteq \left( e^{2t \im F} \left( WF_s(u) \cap \Ker (\re F) \right) \right) 
\cap \Ker (\re F) .
\end{aligned}
\end{equation*}
\end{prop}

\begin{proof}
Let $s \in \ro$. 
As an immediate consequence of Corollary \ref{propagationsing2} we have for $t>0$
\begin{equation*}
WF_s \big(e^{-tq^w(x,D)}u \big) 
\subseteq \left( e^{2t \im F} \left( WF_{s+8d+1}(u) \cap \Ker (\re F) \right) \right) \cap \Ker (\re F).
\end{equation*}
Thus it suffices to show 
\begin{equation}\label{complexcase1}
WF_s(e^{- t q^w(x,D)} u) \subseteq  e^{2 t \, \im F} \, WF_s(u), \quad t \geqslant 0.
\end{equation}

The matrix $T= e^{-2 i t F} \in \Sp(d,\co)$ is positive by \cite[Lemma~4.2]{Rodino2}, and therefore also 
$\overline T^{-1} =  e^{-2 i t \overline{F}} \in \Sp(d,\co)$ is positive (cf. \cite[p.~446]{Hormander2}).
The assumption $[\re F, \im F] = 0$ imply 
\begin{equation*}
e^{-2 i t F} e^{-2 i t \overline{F}} = e^{-4 i t \re F}
\end{equation*}
and $e^{-4 i t \re F} \in  \Sp(d,\co)$ is positive by \cite[Proposition~5.9]{Hormander2}. 
Since $t \geqslant 0$ is arbitrary we may conclude that $e^{-2 i t \re F} \in  \Sp(d,\co)$ is positive.

From these considerations we also see that $e^{2 i t \overline{F}} \in \Sp(d,\co)$ which implies 
\begin{equation*}
e^{-2 i t F} e^{2 i t \overline{F}}  = e^{-2 i t (F-\overline{F})} = e^{4 t \im F} \in \Sp(d,\co). 
\end{equation*}
Since $e^{4 t \im F}$ is a real matrix we have $e^{2 t \im F} \in \Sp(d,\ro)$, again replacing $t$ by $t/2$. 
Since an $\Sp(d,\ro)$ matrix is trivially positive, we have now showed that both $e^{-2 i t \re F} \in  \Sp(d,\co)$ and $e^{2 t \im F} \in \Sp(d,\ro)$
are positive. 

We have, again from $[\re F, \im F] = 0$,
\begin{equation*}
e^{-2 i t F} = e^{2t (\im F - i \re F)} = e^{2 t \im F} e^{-2 t i \re F} = e^{-2 t i \re F} e^{2 t \im F}. 
\end{equation*}

By \cite[Proposition~5.9 and Theorem~5.12]{Hormander2}, from the positivity of $e^{-2 i t \re F} \in  \Sp(d,\co)$ and $e^{2 t \im F} \in \Sp(d,\ro)$ we now obtain
\begin{equation*}
e^{-t q^w(x,D)} = \pm e^{-t (\re q)^w(x,D)} e^{-t (i \im q)^w(x,D)}
\end{equation*}
where the sign ambiguity is related to the corresponding feature of the metaplectic group (see \cite[Proposition~5.9]{Hormander2}). 
The term metaplectic semigroup is used in \cite{Hormander2} to describe this phenomenon. 

To prove \eqref{complexcase1} it suffices therefore by \eqref{realcaseWFs} to show 
\begin{equation}\label{complexsubcase1}
WF_s(e^{-t (\re q)^w(x,D)} u) \subseteq WF_s(u), \quad u \in \cS'(\rr d),  
\end{equation}
that is, microlocality of the operator $e^{-t (\re q)^w(x,D)}$ with respect to the $s$-Gabor wave front set. 

The operator $e^{-t (\re q)^w(x,D)}$ is the solution operator to the initial value problem
\begin{equation*}
\left\{
\begin{array}{rl}
\partial_t u(t,x) + (\re q)^w(x,D) u (t,x) & = 0, \\
u(0,\cdot) & = u_0.
\end{array}
\right.
\end{equation*}
According to F. ~Nicola's result \cite[Theorem~1.2]{Nicola1}, for any $t \geqslant 0$ the operator $e^{-t (\re q)^w(x,D)}$ is a pseudodifferential operator with Weyl symbol $a_t \in S_{0,0}^0$, that is $e^{-t (\re q)^w(x,D)} = a_t^w(x,D)$. Appealing to Proposition \ref{microlocal1} we may conclude that the inclusion \eqref{complexsubcase1} indeed holds. 
\end{proof}

\section{Some particular equations}\label{secequations}

In this section we look at propagation of the $s$-Gabor wave front set according to Proposition \ref{propagationsing3} 
in some particular cases of quadratic forms $q$. 

\subsection{A generalization of \eqref{example3sa}}

Consider the equation
\begin{equation*}
\partial_t u(t,x) + \la x,A x \ra u(t,x) = 0, \quad t \geqslant 0, \quad x \in \rr d, 
\end{equation*}
where $0 \leqslant A \in \rr {d \times d}$ is symmetric. 
It is a particular case of the general equation \eqref{schrodeq} where the quadratic form $q(x,\xi) = \la x, A x \ra$ 
is defined by the matrix
\begin{equation*}
Q = \left(
\begin{array}{cc}
A & 0 \\
0 & 0
\end{array}
\right),
\end{equation*}
and the Hamilton matrix is
\begin{equation*}
F = \J Q = \left(
\begin{array}{cc}
0 & 0 \\
-A & 0
\end{array}
\right). 
\end{equation*}
The solution operator is 
\begin{equation*}
e^{-t q^w(x,D)} u (x) 
= e^{- t \la x, A x \ra} u(x). 
\end{equation*}
Clearly $[\re F,\im F]=0$ and $\Ker (\re F) = \Ker A \times \rr d$.
Proposition \ref{propagationsing3} yields for $s \in \ro$
\begin{equation}\label{propagation1}
WF_s ( e^{- t \la \cdot, A \cdot \ra}  u) 
\subseteq  WF_s (u) \cap ( \Ker A \times \rr d ), \quad t > 0, \quad u \in \cS'(\rr d). 
\end{equation}

This example gives a generalization of \eqref{example3sa} as follows. 
Let $t=1$ and $u=1$. By \eqref{propagation1} and \eqref{example2sa}
\begin{equation}\label{propagation1a}
WF_s ( e^{- \la \cdot, A \cdot \ra} ) 
\subseteq \Ker A \setminus 0 \times \{ 0 \}, \quad s > 0. 
\end{equation}
If $A$ is invertible the opposite inclusion is trivial, and it holds also if $A$ is singular by the following argument. 
Let $x \in \Ker A \setminus 0$ and let $\fy(y) = e^{-|y|^2}$ for $y \in \rr d$. 
The STFT evaluated at $(a x, 0) \in T^* \rr d$ is then for any $a>0$
\begin{align*}
V_\fy (e^{- \la \cdot, A \cdot \ra} )( a x, 0)
& = \int_{\rr d} e^{- \la y, A y \ra - |y-ax|^2 } \, dy \\
& = \int_{\rr d} e^{- \la ax+ y, A (ax+y) \ra - |y|^2  } \, dy \\
& = \int_{\rr d} e^{- \la y, A y \ra - |y|^2 } \, dy 
\end{align*}
which is a positive constant that does not depend on $a>0$. 
The STFT does therefore not decay like $\eabs{z}^{-s}$, for $z$ in any open cone in $T^* \rr d$ containing $(a x, 0)$, for $s>0$. 
This proves the opposite inclusion to \eqref{propagation1a} so we have
\begin{equation}\label{propagation1b}
WF_s ( e^{- \la \cdot, A \cdot \ra} ) 
= \Ker A \setminus 0 \times \{ 0 \}, \quad s>0.
\end{equation}
We also have 
\begin{equation}\label{propagation1s}
WF_s ( e^{- \la \cdot, A \cdot \ra} ) 
= \emptyset, \quad s \leqslant 0, 
\end{equation}
since 
\begin{equation*}
|V_\fy (e^{- \la \cdot, A \cdot \ra} )( x, \xi)|
\lesssim 1, \quad (x,\xi) \in \rr {2d}, \quad \fy \in \cS(\rr d) \setminus 0.  
\end{equation*}

Now let $A \in \cc {d \times d}$ be symmetric with $\im A \geqslant 0$. 
Considered a multiplication operator we have  $e^{i \la x, \re A x \ra/2 } = \mu(\chi)$
where 
\begin{equation*}
\chi = 
\left(
\begin{array}{ll}
I & 0 \\
\re A  & I 
\end{array}
\right) \in \Sp(d,\ro)
\end{equation*}
(cf. \eqref{symplecticoperator}).  
We obtain from Lemma \ref{symplecticGabors} and \eqref{propagation1b}
\begin{equation*}
\begin{aligned}
WF_s (  e^{i \la \cdot, A \, \cdot \ra/2 }  ) 
& = WF_s (  e^{i \la \cdot, \re A \, \cdot \ra/2 } e^{- \la \cdot, \im A \, \cdot \ra/2 } ) \\
& = \chi WF_s ( e^{- \la \cdot, \im A \, \cdot \ra/2 } ) \\
& = \{ (x, \re A \, x + \xi):  \, (x,\xi)  \in WF_s ( e^{- \la \cdot, \im A \, \cdot \ra/2 } )  \} \\
& = \{ (x, \re A \,x ):  \, x \in \rr d \cap \Ker (\im A) \setminus 0 \}, \quad s>0, 
\end{aligned}
\end{equation*}
which is the announced generalization of \eqref{example3sa}. 
We also obtain
\begin{equation*}
WF_s (  e^{i \la \cdot, A \, \cdot \ra/2 }  ) = \emptyset, \quad s \leqslant 0. 
\end{equation*}

\subsection{The heat equation}

The heat equation 
\begin{equation*}
\partial_t u(t,x) - \Delta_x u(t,x) = 0, \quad t \geqslant 0, \quad x \in \rr d, 
\end{equation*}
is a particular case of the general equation \eqref{schrodeq} where $q(x,\xi) = |\xi|^2$ is defined by the matrix
\begin{equation*}
Q = \left(
\begin{array}{cc}
0 & 0 \\
0 & I
\end{array}
\right) \in \rr {2d \times 2d},
\end{equation*}
with Hamilton matrix
\begin{equation*}
F = \J Q = \left(
\begin{array}{cc}
0 & I \\
0 & 0
\end{array}
\right). 
\end{equation*}

Since $[\re F, \im F] = 0$ and $\Ker(\re F) = \rr d \times \{0\}$ Proposition \ref{propagationsing3} gives
for any $s \in \ro$
\begin{equation}\label{heatinclusion}
WF_s(e^{-t q^w(x,D)} u) 
\subseteq  WF_s(u) \cap ( \rr d \times \{ 0 \} ), \quad t > 0, \quad u \in \cS'(\rr d). 
\end{equation}

Therefore, if $s \in \ro$, $u \in \cS'(\rr d)$ and
\begin{equation}\label{WFsassumption1}
WF_s(u) \cap ( \rr d \times \{ 0 \} ) = \emptyset 
\end{equation}
then $WF_s (e^{- t q^w(x,D)} u) = \emptyset$, 
that is 
\begin{equation}\label{WFsconclusion1}
|V_\fy (e^{- t q^w(x,D)} u)(z) | 
\lesssim \eabs{z}^{-s}, \quad z \in \rr {2d}, \quad t>0, 
\end{equation}
where $\fy \in \cS(\rr d) \setminus 0$. 
The assumption \eqref{WFsassumption1} means that 
there exists an open conic set $\Gamma \subseteq T^* \rr d \setminus 0$ such that 
\begin{equation}\label{WFsassumption2}
(\rr d \setminus 0) \times \{0\} \subseteq \Gamma, \quad
\sup_{z \in \Gamma} \eabs{z}^{s} |V_\fy u(z) | 
\lesssim 1. 
\end{equation}
Thus the mild regularity assumption in a conic neighborhood of the phase space directions $(\rr d \setminus 0) \times \{0\}$ on the initial datum \eqref{WFsassumption2} gives the global (isotropic) phase space conclusion \eqref{WFsconclusion1} on the solution $e^{- t q^w(x,D)} u$ for all $t>0$. 
There is an immediate regularizing effect of the heat propagator provided the initial datum has some regularity in the directions $(\rr d \setminus 0) \times \{0\}$ in phase space. 

Note that (cf. \eqref{example2sa})
\begin{equation*}
(\rr d \setminus 0) \times \{ 0 \} = WF_s ( 1 ) 
\end{equation*}
for $s>0$. 
When $s>0$ the assumption \eqref{WFsassumption1} thus means that $WF_s (u)$ is disjoint from the $s$-Gabor wave front set of the function $1$, which is an eigenfunction for the heat propagator. 
If the initial datum is the function $1$, the $s$-Gabor wave front set for the solution remains the same for all $t>0$, and if the initial datum has $s$-Gabor wave front set disjoint from that of the function $1$, then regularization of the solution occurs. 

The inclusion \eqref{heatinclusion} is a refinement of the corresponding result for the Gabor wave front set, see\cite[Eq.~(6.6)]{Rodino2}. 



\begin{thebibliography}{2000}

\bibitem{Carypis1}
E.~Carypis and P.~Wahlberg, \textit{Propagation of exponential phase space singularities for Schr\"odinger equations with quadratic Hamiltonians}, 
arXiv:1510.00325 [math.AP], 2015. 

\bibitem{Cordero5}
E.~Cordero, F.~Nicola and L.~Rodino,
\textit{Propagation of the Gabor wave front set for Schr\"odinger equations with non-smooth potentials}, 
Rev. Math. Phys. \textbf{27} (1) (33 pp.), 2015.

\bibitem{Cordero6}
E.~Cordero and F.~Nicola,
\textit{On the Schr\"odinger equation with potential in modulation spaces},
J. Pseudo-Differ. Oper. Appl. \textbf{5}, 319--341, 2014. 

\bibitem{Feichtinger1} 
H.~G.~Feichtinger, \emph{Modulation spaces on
locally compact abelian groups}, Technical report, University of
Vienna, Vienna, 1983; also in: M. Krishna, R. Radha,
S. Thangavelu (Eds), Wavelets and their applications, Allied
Publishers Private Limited, NewDehli Mumbai Kolkata Chennai Hagpur
Ahmedabad Bangalore Hyderbad Lucknow, 2003, pp. 99--140.

\bibitem{Folland1}
G.~B.~Folland, \textit{Harmonic Analysis in Phase Space}, Princeton University Press, 1989.

\bibitem{Grochenig1}
K.~Gr\" ochenig, \textit{Foundations of Time-Frequency Analysis}, Birkh\" auser, Boston, 2001.

\bibitem{Grochenig2}
\bysame, \textit{Time-frequency analysis of Sj\"ostrand's class}, Revista Mat. Iberoam. \textbf{22} (2), 703--724, 2006.

\bibitem{Hitrik1}
M.~Hitrik and K.~Pravda--Starov,
\textit{Spectra and semigroup smoothing for non-elliptic quadratic operators}, Math. Ann. \textbf{344}, 801--846, 2009.

\bibitem{Hitrik2} 
\bysame, 
\textit{Semiclassical hypoelliptic estimates for non-selfadjoint operators with double characteristics}, Comm. Partial Differential Equations \textbf{35} (6), 988--1028, 2010.

\bibitem{Hitrik3} 
\bysame, 
\textit{Eigenvalues and subelliptic estimates for non-selfadjoint semiclassical operators with double characteristics}, Ann. Inst. Fourier \textbf{63} (3), 985--1032, 2013.

\bibitem{Holst1}
A.~Holst, J.~Toft and P.~Wahlberg, \textit{Weyl product algebras and modulation spaces},
J. Funct. Anal. \textbf{251}, 463--491, 2007.

\bibitem{Hormander0}
L.~H\"ormander, \textit{The Analysis of Linear Partial Differential Operators}, vol I, III,
Springer-Verlag, Berlin Heidelberg NewYork Tokyo, 1983.

\bibitem{Hormander1}
\bysame, \textit{Quadratic hyperbolic operators}, Microlocal Analysis and Applications, Lecture Notes in Math. \textbf{1495}, Eds. L. Cattabriga, L. Rodino, pp. 118--160, Springer, 1991.

\bibitem{Hormander2}
\bysame, \textit{Symplectic classification of quadratic forms, and general Mehler formulas}, Math. Z. \textbf{219}, 413--449, 1995.  

\bibitem{Nicola1}
F.~Nicola, \textit{Phase space analysis of semilinear parabolic equations}, J. Funct. Anal. \textbf{267} (3), 727--743, 2014. 

\bibitem{Nicola2}
F.~Nicola and L.~Rodino,
\textit{Propagation of Gabor singularities for semilinear Schr\"odinger equations}, 
Nonlinear Differ. Equ. Appl. \textbf{22} (6), 1715--1732, 2015.  

\bibitem{Pravda-Starov1}
K.~Pravda--Starov,
\textit{Contraction semigroups of elliptic quadratic differential operators}, Math. Z. \textbf{259}, 363--391, 2008.

\bibitem{Pravda-Starov2} 
\bysame, \textit{Subelliptic estimates for quadratic differential operators}, Amer. J. Math. \textbf{133} (1), 39--89, 2011. 

\bibitem{Rodino1}
L.~Rodino and P.~Wahlberg, \textit{The Gabor wave front set}, Monatsh. Math. \textbf{173}, 625--655, 2014. 

\bibitem{Rodino2}
K.~Pravda--Starov, L.~Rodino and P.~Wahlberg, \textit{Propagation of Gabor singularities for Schr\"odinger equations with quadratic Hamiltonians}, arXiv:1411.0251 [math.AP], 2015. 

\bibitem{Schulz1}
R.~Schulz and P.~Wahlberg, \textit{Microlocal properties of Shubin pseudodifferential and localization operators}, 
J. Pseudo-Differ. Oper. Appl. \textbf{7} (1), 91---111, 2016. 

\bibitem{Shubin1}
M.~A.~Shubin, \textit{Pseudodifferential Operators and Spectral Theory}, Springer, 2001. 

\bibitem{Sjostrand1}
J.~Sj\"ostrand, \textit{An algebra of pseudodifferential operators},
Math. Res. L. \textbf{1}, 185--192, 1994.

\bibitem{Sjostrand2}
\bysame, \textit{Wiener type algebras of pseudodifferential operators},
S\'eminaire Equations aux D\'eriv\'ees Partielles, Ecole Polytechnique, 1994/1995,
Expos\'e n$^{\circ}$ IV.

\bibitem{Taylor1}
M.~E.~Taylor, \textit{Noncommutative Harmonic Analysis}, Mathematical Surveys and Monographs \textbf{22}, AMS, Providence, Rhode Island, 1986.

\bibitem{Viola1} 
J.~Viola, 
\textit{Non-elliptic quadratic forms and semiclassical estimates for non-selfadjoint operators}, Int. Math. Res. Notices \textbf{20}, 4615--4671, 2013. 

\bibitem{Viola2} 
\bysame, 
\textit{Spectral projections and resolvent bounds for partially elliptic quadratic differential operators}, J. Pseudo-Differ. Oper. Appl. \textbf{4}, 145--221, 2013. 

\bibitem{Yosida1}
K.~Yosida, \textit{Functional Analysis}, Classics in Mathematics, Springer-Verlag, Berlin Heidelberg, 1995.

\end{thebibliography}
\end{document}